\documentclass{amsart}

\usepackage{amsthm}
\usepackage[usenames,dvipsnames]{xcolor}
\usepackage{graphicx}
\usepackage{stmaryrd} %has \mapsfrom, \llbracket, \rrbracket
\usepackage{tikz}
\usepackage{tikz-cd}
\usepackage{ bbold }
\usetikzlibrary{arrows}
\usepackage{verbatim}
\usepackage{amssymb,amsfonts}
\usepackage[all,arc]{xy}
\usepackage{enumerate}
\usepackage{mathrsfs, mathabx}
\usepackage{xcolor}[2007/01/21]
\usepackage{hyperref}
\usepackage{soul}
\usepackage{esint} % \varointctrclockwise
\usepackage{multirow} % \multirow (for a simple table)
\usepackage[all]{xy}
\usepackage[margin=1.0in]{geometry}
\usepackage[
   open,
   openlevel=2,
   atend
 ]{bookmark}[2011/12/02]
\usepackage{graphics}

\bookmarksetup{color=blue}

\theoremstyle{definition} 
\newtheorem{thm}{Theorem}[section]

\newtheorem{lem}[thm]{Lemma}
\newtheorem{conj}[thm]{Conjecture}

\theoremstyle{definition}

\newtheorem{rmk}[thm]{Remark}

\theoremstyle{definition}

\theoremstyle{remark}

\newcommand{\GL}{\mathrm{GL}}

\newcommand{\Aut}{\mathrm{Aut}}

\newcommand{\bZ}{\mathbb{Z}}

\newcommand{\bF}{\mathbb{F}}

\newcommand{\mr}{\mathrm}
\newcommand{\mf}{\mathfrak}

\newcommand{\Mat}{\mathrm{Mat}}
\newcommand{\Prob}{\mathrm{Prob}}

\newcommand{\rk}{\mr{rank}}

\newcommand{\id}{\mathrm{id}}
\newcommand{\bs}{\boldsymbol}

\newcommand{\ra}{\rightarrow}

\newcommand{\lra}{\leftrightarrow}

\newcommand{\op}{\oplus}

\newcommand{\cok}{\mathrm{cok}}

\newcommand{\be}{\begin{enumerate}}
\newcommand{\ee}{\end{enumerate}}

\numberwithin{equation}{section}

\begin{document}

\title[Cokernels of $p$-adic matrices]{Generalizations of results of Friedman and Washington on cokernels of random $p$-adic matrices}
\date{\today}
\author{Gilyoung Cheong}
\address{Department of Mathematics, University of California, Irvine, 340 Rowland Hall, Irvine, CA 92697}
\email{gilyounc@uci.edu}

\author{Nathan Kaplan}
\address{Department of Mathematics, University of California, Irvine, 340 Rowland Hall, Irvine, CA 92697}
\email{nckaplan@uci.edu}

\begin{abstract}
Let $p$ be prime and $X$ be a Haar-random $n \times n$ matrix over $\bZ_{p}$, the ring of $p$-adic integers. Let $P_{1}(t), \dots, P_{l}(t) \in \bZ_{p}[t]$ be monic polynomials of degree at most $2$ whose images modulo $p$ are distinct and irreducible in $\bF_{p}[t]$.  For each $j$, let $G_{j}$ be a finite module over $\bZ_{p}[t]/(P_{j}(t))$. We show that as $n$ goes to infinity, the probabilities that $\cok(P_{j}(X)) \simeq G_{j}$ are independent, and each probability can be described in terms of a Cohen--Lenstra distribution. We also show that for any fixed $n$, the probability that $\cok(P_{j}(X)) \simeq G_{j}$ for each $j$ is a constant multiple of the probability that that $\cok(P_{j}(\bar{X})) \simeq G_{j}/pG_{j}$ for each $j$, where $\bar{X}$ is an $n \times n$ uniformly random matrix over $\bF_{p}$. These results generalize work of Friedman and Washington and prove new cases of a conjecture of Cheong and Huang.
\end{abstract}

\maketitle

\section{Introduction}

Throughout this paper, let $p$ be a prime. For a commutative ring $R$, let $\Mat_n(R)$ denote the set of $n \times n$ matrices with entries in $R$ and let $I_{n}$ denote the $n \times n$ identity matrix.  The Haar measure on the additive group $\Mat_{n}(\bZ_{p}) = \bZ_{p}^{n^{2}}$, with respect to its compact $p$-adic topology, allows one to choose a random matrix $X \in \Mat_{n}(\bZ_{p})$. Let $\cok(X)$ denote the cokernel of $X$.  In \cite{FW}, Friedman and Washington proved that the probability that $\cok(X)$ is isomorphic to a fixed finite abelian $p$-group $G$ converges to $|\Aut(G)|^{-1}\prod_{i=1}^{\infty}(1 - p^{-i})$ as $n \rightarrow \infty$. In particular, this probability is inversely proportional to the size of the automorphism group of $G$. For odd $p$, this probability is the one given in an influential conjecture of Cohen and Lenstra on the distribution of $p$-parts of class groups of imaginary quadratic fields, first introduced in \cite{CL}. Motivated by function field analogues of the Cohen--Lenstra conjecture, Friedman and Washington also proved that
\[
\lim_{n \ra \infty}\underset{X \in \GL_{n}(\bZ_{p})}{\Prob}(\cok(X - I_n) \simeq G) 
= \frac{1}{|\Aut(G)|} \prod_{i=1}^{\infty}(1 - p^{-i}),
\]
that is,
\[
\lim_{n \ra \infty}\underset{X \in \Mat_{n}(\bZ_{p})}{\Prob}\begin{pmatrix} \cok(X) = 0, \\ \cok(X - I_n) \simeq G
\end{pmatrix} = \left(\lim_{n \ra \infty}\underset{X \in \Mat_{n}(\bZ_{p})}{\Prob} (\cok(X) = 0) \right) \left(\lim_{n \ra \infty}\underset{X \in \Mat_{n}(\bZ_{p})}{\Prob} (\cok(X - I_n) \simeq G)\right).
\]
Our first main theorem is a generalization of these results.  For a commutative ring $R$ and an $R$-module $G$, let $\Aut_R(G)$ denote the group of $R$-linear automorphisms of $G$.

\begin{thm}\label{main} 
Let $P_{1}(t), \dots, P_{l}(t) \in \bZ_{p}[t]$ be monic polynomials of degree at most $2$ whose images modulo $p$ are distinct and irreducible in $\bF_{p}[t]$.  For each $j$, let $G_{j}$ be a finite module over $\bZ_{p}[t]/(P_{j}(t))$. We have
\[
\lim_{n \ra \infty}\underset{X \in \Mat_{n}(\bZ_{p})}{\Prob}\begin{pmatrix} \cok(P_{j}(X)) \simeq G_{j} \\ \text{for } 1 \leq j \leq l
\end{pmatrix} = \prod_{j=1}^{l}\frac{1}{|\Aut_{\bZ_{p}[t]/(P_{j}(t))}(G_{j})|}\left(\prod_{i=1}^{\infty}\big(1 - p^{-i\deg(P_{j})}\big)\right).
\]
\end{thm}

This result may be surprising to the reader because if we take $l=2$ with $P_{1}(t) = t$ and $P_{2}(t) = t - 1$, then for any $n$, many events regarding the matrices $P_{1}(X) = X$ and $P_{2}(X) = X - I_{n}$ are dependent as the entries of $X$ completely determine the entries of $X - I_{n}$ and vice versa. Nevertheless, Theorem \ref{main} shows that, for example, the event $\cok(X) \simeq \bZ/p\bZ$ becomes independent from the event $\cok(X-I_{n}) \simeq \bZ/p\bZ$ as $n \ra \infty$. Theorem \ref{main} also proves many new cases of a conjecture of Cheong and Huang \cite[Conjecture 2.3]{CH}. We note that the conjecture needs to be slightly modified from their version, as explained below. In \cite[Theorem C]{CH}, Cheong and Huang proved Theorem \ref{main} when $G_{1} = \cdots = G_{l-1} = 0$ and $\deg(P_l) = 1$, so our result is a significant improvement of theirs.

\begin{conj}[cf. {\cite[Conjecture 2.3]{CH}}]\label{conj} 
The conclusion of Theorem \ref{main} holds without specifying any conditions on the degrees of $P_{1}(t), \dots, P_{l}(t) \in \bZ_{p}[t]$.
\end{conj}

\begin{rmk} Conjecture \ref{conj} is stated slightly differently in \cite[Conjecture 2.3]{CH}.  In that version, each module $G_{j}$ over $\bZ_{p}[t]/(P_{j}(t))$ is only assumed to be a finite abelian $p$-group. We note that $\cok(P_{j}(X))$ has a $\bZ_{p}[t]/(P_{j}(t))$-module structure, where the action of $\bar{t}$ is given via left multiplication by $X$. This implies that some finite abelian $p$-groups do not arise as $\cok(P_{j}(X))$ for any $X$.  For example, if $\deg(P_{j}) > 1$, then $\cok(P_{j}(\bar{X}))$ is a vector space over $\bF_{p}[t]/(P_{j}(t))$, so $\dim_{\bF_{p}}(\cok(P_{j}(\bar{X})))$ is a multiple of $\deg(P_{j})$, where $\bar{X} \in \Mat_{n}(\bF_{p})$ is the image of $X$ modulo $p$. Therefore, we see that $\cok(P_{j}(X))$, considered as a finite abelian $p$-group, cannot be isomorphic to $\bZ/p\bZ$.  As noted in \cite[Remark 2.2]{CH} or \cite[Example 5.9]{CL}, there is a \emph{Cohen-Lenstra distribution} on the set of isomorphism classes of finite modules over any DVR whose residue field is finite, that is, a distribution in which each module appears with frequency inversely proportional to its number of automorphisms. One may check that $\bZ_{p}[t]/(P_{j}(t))$ is indeed a DVR with its unique maximal ideal generated by $p$ and its residue field is isomorphic to $\bF_{p}[t]/(P_{j}(t))$.  We see that Conjecture \ref{conj} is a natural correction of \cite[Conjecture 2.3]{CH}. This change affects only Conjecture 2.3 in \cite{CH}, not any theorems in that paper.

Let $G$ and $G'$ be finite modules over $\bZ_{p}[t]/(P_{j}(t))$. Since $\bZ_{p}[t]/(P_{j}(t))$ is a PID, these modules have a specific type of structure.  We can check that $G$ and $G'$ are isomorphic as modules over $\bZ_{p}[t]/(P_{j}(t))$ if and only if they are isomorphic as finite abelian $p$-groups. We will use this observation in our proofs without mentioning it again.
\end{rmk}

Theorem \ref{main} follows from the following stronger result that holds for any fixed $n \in \bZ_{\geq 1}$, which is also a generalization of a result of Friedman and Washington in \cite{FW}. Let $P(t) \in \bZ_{p}[t]$ be a monic polynomial whose reduction modulo $p$ is irreducible in $\bF_{p}[t]$ and let $G$ be a finite module over $\bZ_{p}[t]/(P(t))$.  Define
\[
r_{p^{\deg(P)}}(G) := \dim_{\bF_{p^{\deg(P)}}}(G/pG),
\]
where we identify $\bF_{p^{\deg(P)}} = \bF_{p}[t]/(P(t))$. 

\begin{thm}\label{main2} 
Let $P_{1}(t), \dots, P_{l}(t) \in \bZ_{p}[t]$ be monic polynomials of degree at most $2$ whose images modulo $p$ are distinct and irreducible in $\bF_{p}[t]$ and let $q_{j} := p^{\deg(P_{j})}$.  For each $j$, let $G_{j}$ be a finite module over $\bZ_{p}[t]/(P_{j}(t))$.  We have
\[
\underset{X \in \Mat_{n}(\bZ_{p})}{\Prob}\left(\begin{array}{c}
\cok(P_{j}(X)) \simeq G_{j} \\
\text{for } 1 \leq j \leq l
\end{array}\right)
= \left(\prod_{j=1}^{l}\frac{q_{j}^{r_{q_{j}}(G_{j})^{2}}\prod_{i=1}^{r_{q_{j}}(G_{j})}(1 - q_{j}^{-i})^{2}}{|\Aut_{\bZ_{p}[t]/(P_{j}(t))}(G_{j})|}\right)
\underset{\bar{X} \in \Mat_{n}(\bF_{p})}{\Prob}\left(\begin{array}{c}
\cok(P_{j}(\bar{X})) \simeq G_{j}/pG_{j} \\
\text{for } 1 \leq j \leq l
\end{array}\right).
\]
\end{thm}

\begin{rmk}\label{sumranks}
Theorem \ref{main2} is trivial when $n < \sum_{j=1}^{l} \dim_{\bF_{p}}(G_{j}/pG_{j})$ since each side of the equality is $0$. This follows from the discussion of the basics of the $\bF_q[t]$-module structure of a matrix $\bar{X} \in \Mat_n(\bF_q)$ given at the start of Section \ref{proof_outline} and the fact that $\cok(P_j(X)) \pmod{p} \simeq \cok(P_j(\bar{X}))$.
\end{rmk}

Theorem \ref{main2} follows from the following counting result for matrices in $\Mat_{n}(\bZ/p^{N+1}\bZ)$ with a fixed reduction modulo $p$.

\begin{thm}\label{main3} 
Assume the notation and hypotheses in Theorem \ref{main2}.
Fix any $\bar{X} \in \Mat_{n}(\bF_{p})$ such that for each $1 \leq j \leq l$, we have
\[
\dim_{\bF_{q_{j}}}\left(\cok(P_{j}(\bar{X}))\right) = r_{q_{j}}(G_{j}).
\]
Choose any $N \in \bZ_{\geq 0}$ such that $p^{N}G_{j} = 0$ for $1 \leq j \leq l$.  Then
\[
\#\left\{\begin{array}{c}
X \in \Mat_{n}(\bZ/p^{N+1}\bZ) \colon \\
\cok(P_{j}(X)) \simeq G_{j} \\
\text{for } 1 \leq j \leq l\\
\text{and } X \equiv \bar{X} \pmod{p}
\end{array}\right\} = p^{Nn^{2}}\prod_{j=1}^{l}\frac{q_{j}^{r_{q_{j}}(G_{j})^{2}}\prod_{i=1}^{r_{q_{j}}(G_{j})}(1 - q_{j}^{-i})^{2}}{|\Aut_{\bZ_{p}[t]/(P_{j}(t))}(G_{j})|}.
\]
In particular, the left-hand side does not depend on the choice of $\bar{X} \in \Mat_{n}(\bF_{p})$.
\end{thm}

\begin{conj}\label{conj2} 
The conclusion of Theorem \ref{main2} holds without specifying any conditions on the degrees of $P_{1}(t), \dots, P_{l}(t)$.
\end{conj}

\begin{conj}\label{conj3} 
The conclusion of Theorem \ref{main3} holds without specifying any conditions on the degrees of $P_{1}(t), \dots, P_{l}(t)$.
\end{conj}

\begin{rmk} 
Conjecture \ref{conj2} implies Conjecture \ref{conj} in the same way that Theorem \ref{main2} implies Theorem \ref{main}, and Conjecture \ref{conj3} implies Conjecture \ref{conj2} in the same way that Theorem \ref{main3} implies to Theorem \ref{main2}. In recent communication with Jungin Lee, we were surprised to learn that Conjecture \ref{conj} can be proven with a different method, which is to appear in Lee's upcoming work \cite{LeePaper}. However, Lee's argument does not prove Theorem \ref{main2} or Theorem \ref{main3}, and Conjectures \ref{conj2} and \ref{conj3} remain open.
\end{rmk}

Friedman and Washington prove the special case of Theorem \ref{main3} where $l=1$ and $\deg(P_{1}) = 1$. Our proof of Theorem \ref{main3} is based on theirs but involves additional inputs related to the Smith normal form and the minors of a matrix. We study the conditions on the entries of a matrix over $\bZ/p^{N+1}\bZ$ that determine whether or not its cokernel is isomorphic to a particular finite module $G$.  We then apply elementary operations for block submatrices, which we summarize in Lemma \ref{elem} so that we can apply the $l = 1$ case multiple times to prove Theorem \ref{main3}.

There are several approaches to understanding the distribution of $\cok(X)$ for $X \in \Mat_n(\bZ_p)$ that have appeared since the original result of Friedman and Washington.  One approach that plays a major role in other work on cokernels of families of random $p$-adic matrices is the method of \emph{moments} where one studies the expected number of surjections from $\cok(X)$ to a fixed finite abelian $p$-group.  See \cite{Woo16}, \cite[Section 8]{EVW}, and  \cite{Woo17, Woo19} for more on this perspective.  Evans gives a Markov chain approach to this problem in \cite{Evans}.  Van Peski gives a new approach to this result in his work on cokernels of products of $p$-adic random matrices in \cite{VanPeski}.  It is not immediately clear how to adapt any of these approaches to study cases of Theorem \ref{main} where $l > 1$, or where $l =1$ and $\deg(P_{1}) = 2$.

\section{Theorem \ref{main3} implies Theorem \ref{main2} and Theorem \ref{main2} implies Theorem \ref{main}}\label{implication} 

We begin this section by recalling two key lemmas from \cite{CH}.

\begin{lem}[\cite{CH}, Lemma 4.3] \label{Haar} 
Let $l \in \bZ_{\ge 1}$ and $G_{1}, \dots, G_{l}$ be finite abelian $p$-groups. Choose any $N \in \bZ_{\geq 0}$ such that $p^{N}G_{1} = \cdots = p^{N}G_{l} = 0$. For any monic polynomials $f_{1}(t), \dots, f_{l}(t) \in \bZ_{p}[t]$ and $n \in \bZ_{\geq 1}$, we have
\[
\underset{X \in \Mat_{n}(\bZ_{p})}{\Prob}\left(\begin{array}{c}
\cok(f_{j}(X)) \simeq G_{j} \\
\text{ for } 1 \leq j \leq l
\end{array}\right) = \underset{X \in \Mat_{n}(\bZ/p^{N+1}\bZ)}{\Prob}\left(\begin{array}{c}
\cok(f_{j}(X)) \simeq G_{j} \\
\text{ for } 1 \leq j \leq l
\end{array}\right).
\]
\end{lem}

The next result follows from \cite[Theorem 2.10]{CH} and \cite[Lemma 5.3]{CH}.

\begin{lem}\label{CL} 
Let $l \in \bZ_{\ge 1},\  r_{1}, \dots, r_{l} \in \bZ_{\geq 0}$ and $P_{1}(t), \dots, P_{l}(t) \in \bF_{p}[t]$ be distinct irreducible polynomials. We have
\[
\lim_{n \ra \infty}\underset{\bar{X} \in \Mat_{n}(\bF_{p})}{\Prob}\left(\begin{array}{c}
\dim_{\bF_{p}[t]/(P_{j}(t))}(\cok(P_{j}(\bar{X}))) = r_{j} \\
\text{ for } 1 \leq j \leq l
\end{array}\right) = \prod_{j=1}^{l} \left( \frac{p^{-r_{j}^{2}\deg(P_{j})}\prod_{i=1}^{\infty}(1 - p^{-i\deg(P_{j})})}{\prod_{i=1}^{r_{j}}(1 - p^{-i\deg(P_{j})})^{2}} \right).
\]
\end{lem}

\begin{proof}[Proof that Theorem \ref{main3} implies Theorem \ref{main2} and Theorem \ref{main2} implies Theorem \ref{main}] Throughout the proof, we write $R_{j} := \bZ_{p}[t]/(P_{j}(t))$ and $q_{j} := p^{\deg(P_{j})}$ so that $\bF_{q_{j}} = \bF_{p}[t]/(P_{j}(t))$. By applying Lemma \ref{Haar}, we see that it is enough to prove the desired statements with $\Mat_n(\bZ/p^{N+1}\bZ)$ in place of $\Mat_{n}(\bZ_{p})$. Moreover, as explained in Remark \ref{sumranks}, we may assume that $n \geq \sum_{j=1}^{l}\dim_{\bF_{p}}(\cok(P_{j}(\bar{X}))$.

Theorem \ref{main3} implies that
\[
\#\left\{\begin{array}{c}
X \in \Mat_{n}(\bZ/p^{N+1}\bZ) : \\
\cok(P_{j}(X)) \simeq G_{j} \\
\text{for } 1 \leq j \leq l
\end{array}\right\} 
= p^{Nn^{2}}\left(\prod_{j=1}^{l}\frac{q_{j}^{r_{q_{j}}(G_{j})^{2}}\prod_{i=1}^{r_{q_{j}}(G_{j})}(1 - q_{j}^{-i})^{2}}{|\Aut_{R_{j}}(G_{j})|}\right)
\cdot \#\left\{\begin{array}{c}
\bar{X} \in \Mat_{n}(\bF_{p}) : \\
\cok(P_{j}(\bar{X})) \simeq G_{j}/pG_{j} \\
\text{for } 1 \leq j \leq l
\end{array}\right\}.
\]
Dividing by $p^{(N+1)n^{2}} = \# \Mat_{n}(\bZ/p^{N+1}\bZ)$ and noting that $\#\Mat_{n}(\bF_{p}) = p^{n^{2}}$, we have
\begin{align*}
&  \underset{X \in \Mat_{n}(\bZ/p^{N+1}\bZ)}{\Prob}\left(\begin{array}{c}
\cok(P_{j}(X)) \simeq G_{j} \\
\text{for } 1 \leq j \leq l
\end{array}\right) \\
&= \left(\prod_{j=1}^{l}\frac{q_{j}^{r_{q_{j}}(G_{j})^{2}}\prod_{i=1}^{r_{q_{j}}(G_{j})}(1 - q_{j}^{-i})^{2}}{|\Aut_{R_{j}}(G_{j})|}\right)
\frac{\#\left\{\begin{array}{c}
\bar{X} \in \Mat_{n}(\bF_{p}) : \\
\cok(P_{j}(\bar{X})) \simeq G_{j}/pG_{j} \\
\text{for } 1 \leq j \leq l
\end{array}\right\}}{\#\Mat_{n}(\bF_{p})} \\
& = \left(\prod_{j=1}^{l}\frac{q_{j}^{r_{q_{j}}(G_{j})^{2}}\prod_{i=1}^{r_{q_{j}}(G_{j})}(1 - q_{j}^{-i})^{2}}{|\Aut_{R_{j}}(G_{j})|}\right)
\underset{\bar{X} \in \Mat_{n}(\bF_{p})}{\Prob}\left(\begin{array}{c}
\cok(P_{j}(\bar{X})) \simeq G_{j}/pG_{j} \\
\text{for } 1 \leq j \leq l
\end{array}\right),
\end{align*}
so Theorem \ref{main2} follows.

Next, assume Theorem \ref{main2}. Applying Lemma \ref{CL} with $r_j = r_{q_j}(G_j)$ shows that
\[
\lim_{n \ra \infty}\underset{\bar{X} \in \Mat_{n}(\bF_{p})}{\Prob}\left(\begin{array}{c}
\cok(P_{j}(\bar{X})) \simeq G_{j}/pG_{j} \\
\text{ for } 1 \leq j \leq l
\end{array}\right) = \prod_{j=1}^{l} \left( \frac{q_{j}^{-r_{q_{j}}(G_{j})^{2}}\prod_{i=1}^{\infty}(1 - q_{j}^{-i})}{\prod_{i=1}^{r_{q_{j}}(G_{j})}(1 - q_{j}^{-i})^{2}} \right).
\]
Starting from the statement of Theorem \ref{main2}, applying Lemma \ref{Haar} and then taking $n \ra \infty$ implies that
\begin{align*}
\lim_{n \ra \infty}\underset{X \in \Mat_{n}(\bZ/p^{N+1}\bZ)}{\Prob}\left(\begin{array}{c}
\cok(P_{j}(X)) \simeq G_{j} \\
\text{for } 1 \leq j \leq l
\end{array}\right) &= \prod_{j=1}^{l}\frac{q_{j}^{r_{q_{j}}(G_{j})^{2}}\prod_{i=1}^{r_{q_{j}}(G_{j})}(1 - q_{j}^{-i})^{2}}{|\Aut_{R_{j}}(G_{j})|} \cdot \frac{q_{j}^{-r_{q_{j}}(G_{j})^{2}}\prod_{i=1}^{\infty}(1 - q_{j}^{-i})}{\prod_{i=1}^{r_{q_j}(G_{j})}(1 - q_{j}^{-i})^{2}} \\
&= \prod_{j=1}^{l}\frac{1}{|\Aut_{R_{j}}(G_{j})|}\prod_{i=1}^{\infty}(1 - q_{j}^{-i}),
\end{align*}
so Theorem \ref{main} follows.
\end{proof}

\section{Proof of Theorem \ref{main3} when $l = 1$}\label{sec_main_l1} 

In this section we prove Theorem \ref{main3} when $l = 1$. 

\subsection{Useful Lemmas} When $l = 1$ and $\deg(P_{1}) = 1$, we consider the following more general version of Theorem \ref{main3}.

\begin{lem}\label{l1deg1} 
Let $(R, \mf{m})$ be a complete DVR with finite residue field $R/\mf{m} = \bF_{q}$, let $G$ be a finite $R$-module, and choose any $N \in \bZ_{\geq 0}$ such that $\mf{m}^{N}G = 0$.  For any $\alpha \in R/\mf{m}^{N+1},\ n \in \bZ_{\geq 0}$, and $\bar{X} \in \Mat_{n}(\bF_{q})$ satisfying $\cok(\bar{X} - \bar{\alpha} I_n) \simeq G/\mf{m}G$, where $\bar{\alpha} \in \bF_{q} = R/\mf{m}$ is the image of $\alpha$ modulo $\mf{m}$, we have
\[
\#\left\{\begin{array}{c}
X \in \Mat_{n}(R/\mf{m}^{N+1}) : \\
R^{n}/(X - \alpha I_{n})R^{n} \simeq G \\
\text{and } X \equiv \bar{X} \pmod{\mf{m}}
\end{array}\right\} = q^{Nn^{2}}\frac{q^{r_{q}(G)^{2}}\prod_{i=1}^{r_{q}(G)}(1 - q^{-i})^{2}}{|\Aut_{R}(G)|},
\]
where $r_{q}(G) := \dim_{\bF_{q}}(G/\mf{m}G)$.
\end{lem}

Lemma \ref{l1deg1} can be deduced from the arguments introduced by Friedman and Washington in \cite{FW} although they only discuss the case $R = \bZ_{p}$. In this section, we give a different proof of this result. We need to apply this more general version of Lemma \ref{l1deg1} in our proof of Theorem \ref{main3} when at least one of the polynomials $P_j(t)$ has degree $2$. 

The following lemma, which we learned from Jungin Lee \cite{LeePaper}, is crucial to our proof of Theorem \ref{main3} when at least one of the polynomials $P_j(t)$ has degree $2$.

\begin{lem}[Lee]\label{Lee} 
Given $m \in \bZ_{\geq 0}$, let $P(t) \in (\bZ/p^{m}\bZ)[t]$ be a monic polynomial of degree $d$.  Consider 
\[
R := (\bZ/p^{m}\bZ)[t]/(P(t)) = \bZ/p^{m}\bZ \op \bar{t} (\bZ/p^{m}\bZ) \op \cdots \op \bar{t}^{d-1} (\bZ/p^{m}\bZ).
\]
Fix $X \in \Mat_{n}(\bZ/p^{m}\bZ)$. The map
\[
\psi \colon \frac{(\bZ/p^{m}\bZ)^{n}}{P(X)(\bZ/p^{m}\bZ)^{n}} \ra \cok_{R}(X - \bar{t}I_{n}) := \frac{R^{n}}{(X - \bar{t}I_{n})R^{n}}
\]
defined by $\psi([v]) = [v]$, where $v \in (\bZ/p^{m}\bZ)^{n}$, is an $R$-linear isomorphism.
\end{lem}

\begin{proof}  
Since $P(\bar{t}) = 0$ in $R = (\bZ/p^{m}\bZ)[t]/(P(t))$, we have $P(x) = (x - \bar{t})Q(x)$ for some $Q(x) \in R[x]$. For $w \in (\bZ/p^{m}\bZ)^{n}$, we have
\[
P(X)w = (X - \bar{t}I_{n})Q(X)w
\]
in $R^{n}$, so $\psi$ is well-defined. Since 
\[
\psi(\bar{t}[v]) = \psi([Xv]) = [X v] = [\bar{t}v] = \bar{t}[v],
\]
we see that $\psi$ is $R$-linear.

Suppose $\psi([v]) = 0$ in $R^n/(X - \bar{t}I_{n})R^n$.  Let $v \in (\bZ/p^{m}\bZ)^{n}$ be any representative of $[v]$.  There exist $w_{0}, \dots, w_{d-1} \in (\bZ/p^{m}\bZ)^{n}$ such that when considered as an element of $R^n$,
\begin{align*}
v &= (X - \bar{t}I_{n})(w_{0} + \bar{t}w_{1} + \cdots + \bar{t}^{d-1}w_{d-1}) \\
&= Xw_{0} + \bar{t}Xw_{1} + \dots + \bar{t}^{d-1}Xw_{d-1} - (\bar{t}w_{0} + \bar{t}^{2}w_{1} + \cdots + \bar{t}^{d}w_{d-1}) \\
&= Xw_{0} + \bar{t}(Xw_{1} - w_{0}) + \bar{t}^{2}(Xw_{2} - w_{1}) + \cdots + \bar{t}^{d-1}(Xw_{d-1} - w_{d-2}) - \bar{t}^{d}w_{d-1}.
\end{align*}
Writing $P(t) =  t^{d} + a_{d-1}t^{d-1} + \cdots + a_{1}t + a_{0}$, implies that as an element of $R^n$,
\[
v = Xw_{0} + a_{0}w_{d-1} + \bar{t}(Xw_{1} -  w_{0} + a_{1}w_{d-1}) + \bar{t}^{2}(Xw_{2} - w_{1} + a_{2}w_{d-1}) + \cdots + \bar{t}^{d-1}(Xw_{d-1} - w_{d-2} + a_{d-1}w_{d-1}).
\]
Since $v \in (\bZ/p^{m}\bZ)^{n}$, the decomposition $R = \bZ/p^{m}\bZ \op \bar{t} (\bZ/p^{m}\bZ) \op \cdots \op \bar{t}^{d-1} (\bZ/p^{m}\bZ)$ implies that as elements of $(\bZ/p^{m}\bZ)^{n}$,
\begin{align*}
&v = Xw_{0} + a_{0}w_{d-1}, \\
&w_{0} = Xw_{1} +  a_{1}w_{d-1}, \\
&w_{1} = Xw_{2} +  a_{2}w_{d-1}, \\
&\hspace{1cm} \cdots \\
&w_{d-2} =  Xw_{d-1} +  a_{d-1}w_{d-1}.
\end{align*}
Therefore, as an element of $(\bZ/p^{m}\bZ)^{n}$,
\begin{align*}
v &= Xw_{0} + a_{0}w_{d-1} \\
&= X^{2}w_{1} + a_{1}Xw_{d-1} + a_{0}w_{d-1} \\
&= X^{3}w_{2} + a_{2}X^{2}w_{d-1} + a_{1}Xw_{d-1} + a_{0}w_{d-1} \\
&\hspace{1cm} \cdots \\
&= X^{d-1}w_{d-2} + a_{d-2}X^{d-2}w_{d-1} +  a_{d-3}X^{d-3}w_{d-1} + \cdots + a_{1}Xw_{d-1} + a_{0}w_{d-1} \\
&= X^{d}w_{d-1} + a_{d-1}X^{d-1}w_{d-1} + a_{d-2}X^{d-2}w_{d-1} +  a_{d-3}X^{d-3}w_{d-1}  + \cdots + a_{1}Xw_{d-1} + a_{0}w_{d-1} \\
&= (X^{d} + a_{d-1}X^{d-1} + a_{d-2}X^{d-2} +  a_{d-3}X^{d-3} + \cdots + a_{1}X + a_{0})w_{d-1} = P(X)w_{d-1}.
\end{align*}
This means $[v] = 0$ in $(\bZ/p^{m}\bZ)^{n}/P(X)(\bZ/p^{m}\bZ)^{n}$, and we conclude that $\psi$ is injective.

Given any $v_{0} + \bar{t}v_{1} + \cdots + \bar{t}^{d-1}v_{d-1} \in R^{n}$, where each $v_{i} \in (\bZ/p^{m}\bZ)^{n}$, we have
\[[
v_{0} + Xv_{1} + \cdots + X^{d-1}v_{d-1}] = [v_{0} + \bar{t}v_{1} + \cdots + \bar{t}^{d-1}v_{d-1}]
\]
in $ R^{n}/(X - \bar{t}I_{n})R^{n}$. This shows that $\psi$ is surjective.
\end{proof}

\subsection{Proof of Theorem \ref{main3} when $l = 1$}\label{l1} 
We now use Lemma \ref{l1deg1} to prove Theorem \ref{main3} when $l = 1$.  Let $P(t) \in \bZ_{p}[t]$ be a monic polynomial of degree at most $2$ whose reduction modulo $p$ is irreducible in $\bF_{p}[t]$. Let $q := p^{\deg(P)}$.  We want to show that for any finite module $G$ over $\bZ_{p}[t]/(P(t)),\ N \in \bZ_{\geq 0}$ such that $p^{N+1}G = 0$, and $\bar{X} \in \Mat_{n}(\bF_{p})$ with $r_{q}(G) := \dim_{\bF_{q}}(G/pG) = \dim_{\bF_{q}}(\cok(P(\bar{X})))$, we have
\[
\#\left\{\begin{array}{c}
X \in \Mat_{n}(\bZ/p^{N+1}\bZ) : \\
\cok(P(X)) \simeq G \\
\text{and } X \equiv \bar{X} \pmod{p}
\end{array}\right\} = p^{Nn^{2}}\frac{q^{r_{q}(G)^{2}}\prod_{i=1}^{r_{q}(G)}(1 - q^{-i})^{2}}{|\Aut_{\bZ_{p}[t]/(P(t))}(G)|}.
\]

\begin{proof} 
Lemma \ref{l1deg1} with $R = \bZ_p$ gives the result we need for $\deg(P) = 1$, so we suppose that $\deg(P) = 2$. In this case, we have $q = p^{2}$. For ease of notation, we write $P(t)$ for the image of $P(t)$ in $(\bZ/p^{N+1}\bZ)[t]$. Let $A \in \Mat_{n}(\bF_{p})$ satisfy
\[
\dim_{\bF_{q}}(\cok(P(A))) = r_{q}(G).
\]
For the rest of the proof, let $R = (\bZ/p^{N+1}\bZ)[t]/(P(t))$. Applying Lemma \ref{Lee} shows that 
\[
\#\left\{\begin{array}{c}
X \in \Mat_{n}(\bZ/p^{N+1}\bZ) \colon \\
\cok(P(X)) \simeq G \\
\text{and } X \equiv A \pmod{p}
\end{array}\right\} = \#\left\{\begin{array}{c}
X \in \Mat_{n}(\bZ/p^{N+1}\bZ) \colon \\
\cok_{R}(X - \bar{t}I_{n}) \simeq G \\
\text{and } X \equiv A \pmod{p}
\end{array}\right\}.
\]
We claim that the size of this set is independent of the choice of $A$.

The decomposition $R  = (\bZ/p^{N+1}\bZ) \oplus \bar{t}(\bZ/p^{N+1}\bZ)$ gives a decomposition $\Mat_{n}(R) = \Mat_{n}(\bZ/p^{N+1}\bZ) \oplus \bar{t}\Mat_{n}(\bZ/p^{N+1}\bZ)$.  This decomposition shows that
\[
c_{N,n} := \#\left\{\begin{array}{c}
Z = X + \bar{t}Y \in \Mat_{n}(R)\colon \\
\cok_{R}(Z) \simeq G \\
\text{and } Z \equiv A - \bar{t}I_{n} \pmod{p}
\end{array}\right\} = \#\left\{\begin{array}{c}
(X, Y) \in \Mat_{n}(\bZ/p^{N+1}\bZ)^{2} \colon \\
\cok_{R}(X + \bar{t}Y) \simeq G, \\
X \equiv A \text{ and } Y \equiv -I_{n} \pmod{p} 
\end{array}\right\}.
\]
By Lemma \ref{l1deg1}, this expression is independent of the choice of $A$. We have
\[
c_{N,n} = \#\left\{\begin{array}{c}
(X, Y) \in \Mat_{n}(\bZ/p^{N+1}\bZ)^{2} \colon \\
\cok_{R}(X + \bar{t}Y) \simeq G, \\
X \equiv A \text{ and } Y \equiv -I_{n} \pmod{p} 
\end{array}\right\} 
= \sum_{pM \in p\Mat_{n}(\bZ/p^{N+1}\bZ)}\#\left\{\begin{array}{c}
X \in \Mat_{n}(\bZ/p^{N+1}\bZ) \colon \\
\cok_{R}(X + \bar{t}(pM - I_{n})) \simeq G \\
\text{and } X \equiv A \pmod{p} 
\end{array}\right\}.
\]
For any $pM \in p\Mat_{n}(\bZ/p^{N+1}\bZ)$, we have a bijection
\[
\left\{\begin{array}{c}
X' \in \Mat_{n}(\bZ/p^{N+1}\bZ) \colon \\
\cok_{R}(X' + \bar{t}(pM - I_{n})) \simeq G \\
\text{and } X' \equiv A \pmod{p} 
\end{array}\right\} \lra \left\{\begin{array}{c}
X \in \Mat_{n}(\bZ/p^{N+1}\bZ) \colon \\
\cok_{R}(X - \bar{t}I_{n}) \simeq G \\
\text{and } X \equiv A \pmod{p} 
\end{array}\right\}
\]
given by $X' \mapsto X = -X'(pM - I_{n})^{-1}$.  Since $|p\Mat_{n}(\bZ/p^{N+1}\bZ)| = p^{Nn^{2}}$, we have
\[
c_{N,n} = p^{Nn^{2}}\#\left\{\begin{array}{c}
X \in \Mat_{n}(\bZ/(p^{N+1})) \colon\\
\cok_{R}(X - \bar{t}I_{n}) \simeq G \\
\text{and } X \equiv A \pmod{p} 
\end{array}\right\},
\]
and so 
\[
c_{N,n}p^{-Nn^{2}} = \#\left\{\begin{array}{c}
X \in \Mat_{n}(\bZ/p^{N+1}\bZ) \colon\\
\cok(P(X)) \simeq G \\
\text{and } X \equiv A \pmod{p}
\end{array}\right\}.
\]
The quantity on the right-hand side is independent of the choice of $A$, as claimed. Moreover, Lemma \ref{l1deg1} implies
\[
c_{N,n} = p^{2Nn^{2}}\frac{q^{r_{q}(G)^{2}}\prod_{i=1}^{r_{q}(G)}(1 - q^{-i})^{2}}{|\Aut_{\bZ_{p}[t]/(P(t))}(G)|},
\]
because $q = p^{2}$. Therefore,
\[
\#\left\{\begin{array}{c}
X \in \Mat_{n}(\bZ/(p^{N+1})) \colon\\
\cok(P(X)) \simeq G \\
\text{and } X \equiv A \pmod{p} 
\end{array}\right\} = c_{N,n}p^{-Nn^{2}} =  p^{Nn^{2}}\frac{q^{r_{q}(G)^{2}}\prod_{i=1}^{r_{q}(G)}(1 - q^{-i})^{2}}{|\Aut_{\bZ_{p}[t]/(P(t))}(G)|}.
\]
\end{proof}

\subsection{Outline of the proof of Lemma \ref{l1deg1}}\label{proof_outline} 
In the rest of this section, we prove Lemma \ref{l1deg1}. Without loss of generality, we may assume $\alpha = 0$. We show that given 
\begin{itemize}
	\item a finite $R$-module $G$,
	\item $N \in \bZ_{\geq 0}$ such that $\mf{m}^{N}G = 0$, and
	\item $\bar{X} \in \Mat_{n}(\bF_{q})$ such that $\cok(\bar{X}) \simeq G/\mf{m}G$ as $\bF_{q}$-vector spaces,
\end{itemize}
we have
\begin{equation}\label{main_count}
\#\left\{\begin{array}{c}
X \in \Mat_{n}(R/\mf{m}^{N+1}) \colon\\
\cok(X) \simeq G \\
\text{and } X \equiv \bar{X} \pmod{\mf{m}}
\end{array}\right\} = q^{Nn^{2}}\frac{q^{r_{q}(G)^{2}}\prod_{i=1}^{r_{q}(G)}(1 - q^{-i})^{2}}{|\Aut_{R}(G)|}.
\end{equation}

In order to give the outline of our argument, we recall some linear algebra related to $\bar{X} \in \Mat_{n}(\bF_{q})$.  We can give $\bF_{q}^{n}$ an $\bF_{q}[t]$-module structure by defining the $t$-action as left multiplication by $\bar{X}$ on the $n \times 1$ matrices over $\bF_{q}$. With this structure in mind, we may write $\bar{X}$ to also mean the corresponding $\bF_{q}[t]$-module, namely the $\bF_{q}$-vector space $\bF_{q}^{n}$ together with the action of $\bar{X}$. Given any irreducible polynomial $P(t) \in \bF_{q}[t]$, we have 
\[
\cok(P(\bar{X})) \simeq \ker(P(\bar{X})) \simeq \bar{X}[P^{\infty}]/P\bar{X}[P^{\infty}]
\]
as $\bF_{q}$-vector spaces, where $\bar{X}[P^{\infty}]$ denotes the $P$-part of the $\bF_{q}[t]$-module $\bar{X}$. For ease of notation, throughout the proof we let $r = r_q(G)$.  Since $\cok(\bar{X}) \simeq G/\mf{m}G \simeq \bF_{q}^{r}$, we have
\[
\bar{X}[t^{\infty}] \simeq \bF_{q}[t]/(t^{m_{1}}) \times \cdots \times \bF_{q}[t]/(t^{m_{r}})
\]
as $\bF_{q}[t]$-modules, where $m_1 \ge m_2 \ge \cdots \ge m_r \ge 1$. In other words, the matrix $\bar{X}$ has $r$ Jordan blocks corresponding to the eigenvalue $0$ with sizes $m_{1}, \dots, m_{r}$. 

Our argument is divided into three main steps:
\begin{enumerate}
\item 
We prove that it is enough to show that \eqref{main_count} holds for $\bar{X} \in \Mat_{n}(\bF_{q})$ of the special form
\[
\bar{X} = \begin{bmatrix}
0 & 0 & 0 \\
0 & \id & 0 \\
0 & 0 & \bar{M}
\end{bmatrix},
\]
where
\begin{itemize}
	\item $\id = I_{m_{1} + \cdots + m_{r} - r}$, the $(m_{1} + \cdots + m_{r} - r) \times (m_{1} + \cdots + m_{r} - r)$ identity matrix, and
	\item $\bar{M} \in \GL_{n - (m_{1} + \cdots + m_{r})}(\bF_{q})$.
\end{itemize}

\item For $\bar{X}$ of this form, we prove that the left-hand side of \eqref{main_count} is 
\[
q^{N(n^{2} - r^{2})}\#\{uA \in u\Mat_{r}(R/\mf{m}^{N+1}) \colon\cok(uA) \simeq G\},
\]
where $u$ is a uniformizer of $R$ (i.e., a generator for its maximal ideal, so $\mf{m} = (u) = uR$).

\item We prove that
\[
\#\{uA \in u\Mat_{r}(R/\mf{m}^{N+1}) \colon \cok(uA) \simeq G\}  = \frac{q^{Nr^{2} + r^{2}}\prod_{i=1}^{r}(1 - q^{-i})^{2}}{|\Aut_{R}(G)|}.
\]
\end{enumerate}

\begin{rmk}\label{FW_remark}
Friedman and Washington prove Lemma \ref{l1deg1} in \cite[p.236]{FW}.  Their proof is similar to the one outlined above.  They reduce the statement to the count given in the third main step.  Then they  note that $\cok(uA') \simeq H$ if and only if $\cok(A') \simeq uH$.  Finally, they compute 
\[
\#\{A' \in \Mat_{r}(R/\mf{m}^{N+1}) \colon \cok(A') \simeq uH\}.
\]
using \cite[Proposition 1]{FW}.  

The final part of our argument is longer but works more directly with the entries of the matrices we consider.  In particular, we describe conditions on a matrix over $R/\mf{m}^{N+1}$ that determine whether or not its cokernel is isomorphic to a particular module $G$.  We give a full proof of Lemma \ref{l1deg1} because several of the pieces are important for the proof of the general case of Theorem \ref{main3}.  We also believe that the techniques in our proof may be useful for other problems about cokernels of families of random $p$-adic matrices.
\end{rmk}

We now carry out the first part of our three step strategy.
\begin{proof}[Proof of Lemma \ref{l1deg1}: Step (1)]

By switching rows and columns of $\bar{X}$, there exist $\bar{Q}_{1}, \bar{Q}_{2} \in \GL_{n}(\bF_{q})$ such that
\[
\bar{Q}_{1}\bar{X}\bar{Q}_{2} = \begin{bmatrix}
0 & 0 & 0 \\
0 & \id & 0 \\
0 & 0 & \bar{M}
\end{bmatrix},
\]
where 
\begin{itemize}
	\item $\id$ is the $(m_{1} + \cdots + m_{r} - r) \times (m_{1} + \cdots + m_{r} - r)$ identity matrix, and
	\item $\bar{M} \in \GL_{n - (m_{1} + \cdots + m_{r})}(\bF_{q})$.\footnote{In \cite[p.234, (11)]{FW}, these $\bar{Q}_{1}$ and $\bar{Q}_{2}$ are taken to be inverses of each other, but it is not possible to find such matrices in general. For example, the $2 \times 2$ matrix $\begin{bmatrix}0 & 1 \\ 0 & 0\end{bmatrix}$ is not similar to any matrix of the form $\begin{bmatrix}0 & b \\ 0 & d\end{bmatrix}$ with $d \neq 0$. Nevertheless, this is an easy fix.}
\end{itemize}

Fix lifts $Q_{1}, Q_{2} \in \Mat_{n}(R/\mf{m}^{N+1})$ of $\bar{Q}_{1}, \bar{Q}_{2}$, meaning $Q_{i} \equiv \bar{Q}_{i} \pmod{\mf{m}}$ for $i \in\{1, 2\}$. Since $\bar{Q}_{1}, \bar{Q}_{2} \in \GL_{n}(\bF_{q})$, we have $Q_{1}, Q_{2} \in \GL_{n}(R/\mf{m}^{N+1})$. Fix a lift $M \in \GL_{n - (m_{1} + \cdots + m_{r})}(R/\mf{m}^{N+1})$ of $\bar{M}$. 

For any lift $X \in \Mat_{n}(R/\mf{m}^{N+1})$ of $\bar{X}$, note that $Q_{1}XQ_{2} \in \Mat_{n}(R/\mf{m}^{N+1})$ is a lift of $\bar{Q}_{1}\bar{X}\bar{Q}_{2}$. On the other hand, if $Y \in \Mat_{n}(R/\mf{m}^{N+1})$ is a lift of $\bar{Q}_{1}\bar{X}\bar{Q}_{2}$, then $Q_{1}^{-1}YQ_{2}^{-1} \in \Mat_{n}(R/\mf{m}^{N+1})$ is a lift of $\bar{X}$. This gives a bijection between the lifts of $\bar{X}$ to $\Mat_{n}(R/\mf{m}^{N+1})$ and the lifts of $\bar{Q}_{1}\bar{X}\bar{Q}_{2}$ to $\Mat_{n}(R/\mf{m}^{N+1})$. Hence, the number of lifts $X \in \Mat_{n}(R/\mf{m}^{N+1})$ of $\bar{X}$ is equal to the number of lifts $Y = Q_{1}XQ_{2} \in \Mat_{n}(R/\mf{m}^{N+1})$ of $\bar{Q}_{1}\bar{X}\bar{Q}_{2}$. Since $\cok(Y) = \cok(Q_{1}XQ_{2}) \simeq \cok(X)$, it is enough to count the lifts $Y = Q_{1}XQ_{2} \in \Mat_{n}(R/\mf{m}^{N+1})$ of $\bar{Q}_{1}\bar{X}\bar{Q}_{2}$ with $\cok(Y) \simeq G$.
\end{proof}

\subsection{Elementary operations for block submatrices} 
Before we carry out the second main step of the proof of Lemma \ref{l1deg1}, we recall some material about elementary row and column operations for block submatrices.  

Let $R$ be a commutative ring and $X \in \Mat_n(R)$.  Each of the following three elementary row operations corresponds to left multiplication by a matrix in $\GL_{n}(R)$:
\begin{itemize}
\item Exchange the $i$-th row $X_{(i)}$ with the $j$-th row $X_{(j)}$ for any distinct $i,j \in [1,n]$;
\item Multiply $X_{(i)}$ by a unit in $R$  for any $i \in [1,n]$;
\item Replace $X_{(i)}$ with $X_{(i)} + a X_{(j)}$ for any $a \in R$ and any distinct $i,j\in [1,n]$.
\end{itemize}
Likewise, each of the following three elementary column operations corresponds to right multiplication by a matrix in $\GL_{n}(R)$:
\begin{itemize}
\item Exchange the $i$-th column $X^{(i)}$ with the $j$-th column $X^{(j)}$  for any distinct $i,j \in [1,n]$;
\item Multiply $X^{(i)}$ by a unit in $R$ for any $i\in [1,n]$;
\item Replace $X^{(i)}$ with $X^{(i)} + a X^{(j)}$ for any $a \in R$ and any distinct $i,j \in [1,n]$.
\end{itemize}
Note that elementary (row or column) operations do not change the isomorphism class of $\cok(X)$.

\hspace{3mm} A key technique in the proof Lemma \ref{l1deg1} is an analogous method for elementary operations with block submatrices of $X$. Let $n$ be a positive integer and $n_{1}, \dots, n_{s} \geq 1$ satisfy $n_{1} + \cdots + n_{s} = n$. We subdivide $X \in \Mat_{n}(R)$ into block submatrices where $X_{[i,j]}$ is an $n_{i} \times n_{j}$ matrix over $R$:
\[
X = \begin{bmatrix}
X_{[1,1]} & X_{[1,2]} & \cdots & X_{[1,s-1]} & X_{[1,s]} \\
X_{[2,1]} & X_{[2,2]} & \cdots & X_{[2,s-1]} & X_{[2,s]} \\
\vdots & \vdots & \cdots & \vdots & \vdots \\
X_{[s-1,1]} & X_{[s-1,2]} & \cdots & X_{[s-1,s-1]} & X_{[s-1,s]}  \\
X_{[s,1]} & X_{[s,2]} & \cdots & X_{[s, s-1]} & X_{[s,s]}
\end{bmatrix}.
\]

\begin{lem}[Elementary operations for block submatrices]\label{elem} 
Keeping the notation as above, fix distinct $i,j \in [1,s]$. 
Any of the following three elementary block row operations on $X \in \Mat_{n}(R)$ corresponds to left multiplication of $X$ by a matrix in $\GL_{n}(R)$:
\begin{enumerate}
\item Exchange the $i$-th (block) row $X_{[i]} = [X_{[i,1]}, \cdots, X_{[i,s]}]$ with the $j$-th row $X_{[j]} = [X_{[j,1]}, \cdots, X_{[j,s]}]$; 
\item Multiply $X_{[i]} = [X_{[i,1]}, \cdots, X_{[i,s]}]$ on the left by any $g \in \GL_{n_{i}}(R)$ to get $gX_{[i]} = [gX_{[i,1]}, \cdots, gX_{[i,s]}]$;
\item For any $n_{i} \times n_{j}$ matrix $A$, replace $X_{[i]}$ with $X_{[i]} + A X_{[j]} = [X_{[i,1]} + AX_{[j,1]} , \cdots, X_{[i,s]} + AX_{[j,s]}]$.
\end{enumerate}
Likewise, any of the following three column block row operations on $X \in \Mat_{n}(R)$ corresponds to right multiplication of $X$ by a matrix in $\GL_{n}(R)$:
\begin{enumerate}
\item Exchange the $i$-th (block) column $X^{[i]} = \begin{bmatrix}X_{[1,i]} \\ \vdots \\ X_{[s,i]}\end{bmatrix}$ with the $j$-th column $X^{[j]} = \begin{bmatrix}X_{[1,j]} \\ \vdots \\ X_{[s,j]}\end{bmatrix}$; 
\item Multiply $X^{[i]} = \begin{bmatrix}X_{[1,i]} \\ \vdots \\ X_{[s,i]}\end{bmatrix}$ on the right by any $g \in \GL_{n_{i}}(R)$ to get $X^{[i]}g = \begin{bmatrix}X_{[1,i]}g \\ \vdots \\ X_{[s,i]}g\end{bmatrix}$;
\item For an $n_{j} \times n_{i}$ matrix $A$, replace $X^{[i]}$ with $X^{[i]} + X^{[j]}A =  \begin{bmatrix}X_{[1,i]} + X_{[1,j]}A \\ \vdots \\ X_{[s,i]} + X_{[s,j]}A\end{bmatrix}$.
\end{enumerate}
In particular, the operations above do not change the isomorphism class of $\cok(X)$.
\end{lem}

\begin{proof} 
We note that
\[
[X_{[i,1]}, \cdots, X_{[i,s]}]^{T} = \begin{bmatrix}X_{[i,1]}^{T} \\ \vdots \\ X_{[i,s]}^{T}\end{bmatrix}.
\]
Therefore, the column operations are given by taking the transposes of the row operations and it is enough to prove Lemma \ref{elem} for the block row operations.

The operations (1) and (3) follow directly from the corresponding ones from the usual elementary operations. The operation (2) corresponds to left multiplication by the block diagonal matrix with blocks $I_{n_1}, I_{n_2},\ldots, I_{i-1}, g, I_{i+1},\ldots, I_s$, that is, the matrix that comes from replacing the $[i,i]$-block of the identity matrix with $g$. This finishes the proof.
\end{proof}

\begin{proof}[Proof of Lemma \ref{l1deg1}: Step (2)] 
Suppose that $\bar{X} \in \Mat_{n}(\bF_{q})$ is of the form described in Step (1) of the outline of the proof given in Section \ref{proof_outline}. Recall that $R$ is a complete DVR with maximal ideal $\mf{m}$ and residue field $R/\mf{m} = \bF_{q}$.  Let $u$ be a uniformizer of $R$, so $\mf{m} = (u)$.

Any lift of $\bar{X} \in \Mat_{n}(\bF_{q})$ to $\Mat_{n}(R/\mf{m}^{N+1})$ is of the form
\[
X = \begin{bmatrix}
uA_{1} & uA_{2} & uA_{3} \\
uA_{4} & \id + uA_{5} & uA_{6} \\
uA_{7} & uA_{8} & M + uA_{9}
\end{bmatrix},
\]
where $uA_{1}, uA_{2}, uA_{3}, uA_{4}, uA_{5}, uA_{6}, uA_{7}, uA_{8}, uA_{9}$ are matrices over $R/\mf{m}^{N+1}$ all of whose entries are in $(u)$ such that 
\begin{itemize}
	\item $uA_{1} \in u\Mat_{r}(R/\mf{m}^{N+1})$,
	\item $uA_{5} \in u\Mat_{m_{1} + \cdots + m_{r} - r}(R/\mf{m}^{N+1}),\ \id = I_{m_{1} + \cdots + m_{r} - r}$, and 
	\item $uA_{9} \in u\Mat_{n - (m_{1} + \cdots + m_{r})}(R/\mf{m}^{N+1})$.
\end{itemize}

Choose representatives $\alpha_1,\ldots, \alpha_q \in R/\mf{m}^{N+1}$ for the equivalence classes in $(R/\mf{m}^{N+1})/(\mf{m}/\mf{m}^{N+1}) \simeq R/\mf{m} = \bF_{q}$.  The filtration $R/\mf{m}^{N+1} \supset \mf{m}/\mf{m}^{N+1} \supset \cdots \supset \mf{m}^{N}/\mf{m}^{N+1}$ shows that each element of $R/\mf{m}^{N+1}$ can be written uniquely as $a_0 + a_1 u + \cdots + a_N u^N$, where each $a_i$ is equal to some $\alpha_j$.  So each entry of the matrix $uA_k$ is of the form $a_{1}u + a_{2}u^{2} + \cdots + a_{N}u^{N}$ where each $a_i$ is equal to some $\alpha_j$.  There are  $q^{Nn^{2}}$ total possible choices for the entries of $uA_{1}, uA_{2}, uA_{3}, uA_{4}, uA_{5}, uA_{6}, uA_{7}, uA_{8}, uA_{9}$ over $R/\mf{m}^{N+1}$ if we do not require any condition on $\cok(X)$. We count the number of choices for which $\cok(X) \cong G$.

First, we freely choose $u\bs{A} := (uA_{2}, uA_{3}, uA_{4}, uA_{5}, uA_{6}, uA_{7}, uA_{8}, uA_{9})$ . There are $q^{N(n^{2} - r^{2})}$ possible choices for $u\bs{A}$. We claim that given $u\bs{A}$ and $uA_{1}$ there exist $P_{1,u\bs{A}}$ and $P_{2,u\bs{A}} \in \GL_{n}(R/\mf{m}^{N+1})$, depending on $u\bs{A}$ but not $uA_{1}$, such that
\begin{equation}\label{P1XP2}
P_{1,u\bs{A}}XP_{2,u\bs{A}} = 
\begin{bmatrix}
uA_{uA_{1}, p\bs{A}} & 0 & 0 \\
0 & \id + uB_{u\bs{A}} & 0 \\
0 & 0 & M + uA_{9}
\end{bmatrix},
\end{equation}
where 
\[
uB_{u\bs{A}} = u(A_{5} - uA_{6}(M + uA_{9})^{-1}A_{8}),
\]
and $uA_{uA_{1}, u\bs{A}}$ depends on $uA_1$ and $u\bs{A}$. We prove the existence of these matrices $P_{1,u\bs{A}}, P_{2,u\bs{A}}$ by describing (block) row and column operations that we can apply to $X$, using Lemma \ref{elem}.  Since $\id + uB_{p\bs{A}}$ and $M + uA_{9}$ are invertible modulo $\mf{m}$, they are also invertible as matrices over $R/\mf{m}^{N+1}$. Therefore,
\[
\cok(X) \simeq \cok(P_{1,u\bs{A}}XP_{2,u\bs{A}}) \simeq \cok(uA_{A_{1}, u\bs{A}}).
\]
The sequence of (block) row and column operations that we apply to $X$ makes it clear that the map taking $uA_1$ to $uA_{uA_{1}, u\bs{A}}$ is a bijection from $u\Mat_{r}(R/\mf{m}^{N+1})$ to itself.  Therefore, the number of choices of $uA_{1}, uA_{2}, uA_{3}, uA_{4}, uA_{5}, uA_{6}, uA_{7}, uA_{8}, uA_{9}$ for which $\cok(uA_{uA_{1}, u\bs{A}}) \simeq G$ is equal to 
\[
q^{N(n^{2} - r^{2})}\#\{uA \in u\Mat_{r}(R/\mf{m}^{N+1}) \colon \cok(uA) \simeq G\}.
\]

Given a choice of $u\bs{A}$, we now describe the (block) row and column operations taking $X$ to the matrix on the right-hand side of \eqref{P1XP2}. Applying Lemma \ref{elem}, subtract 
\[
uA_{6}(M + uA_{9})^{-1}[uA_{7}, uA_{8}, M + uA_{9}] = [u^{2}A_{6}(M + uA_{9})^{-1}A_{7}, u^{2}A_{6}(M + uA_{9})^{-1}A_{8}, uA_{6}]
\]
from the second block row of $X$ to get
\[
\begin{bmatrix}
uA_{1} & uA_{2} & uA_{3} \\
u(A_{4} - uA_{6}(M + uA_{9})^{-1}A_{7}) & \id + u(A_{5} - uA_{6}(M + uA_{9})^{-1}A_{8}) & 0 \\
uA_{7} & uA_{8} & M + uA_{9}
\end{bmatrix}.
\]
Next, subtract
\[
uA_{3}(M + uA_{9})^{-1}[uA_{7}, uA_{8}, M + uA_{9}] = [u^{2}A_{3}(M + uA_{9})^{-1}A_{7}, u^{2}A_{3}(M + uA_{9})^{-1}A_{8}, uA_{3}]
\]
from the first block row to get
\[
\begin{bmatrix}
u(A_{1} - uA_{3}(M + uA_{9})^{-1}A_{7}) & u(A_{2} - uA_{3}(M + uA_{9})^{-1}A_{8}) & 0 \\
u(A_{4} - uA_{6}(M + uA_{9})^{-1}A_{7}) & \id + u(A_{5} - uA_{6}(M + uA_{9})^{-1}A_{8}) & 0 \\
uA_{7} & uA_{8} & M + uA_{9}
\end{bmatrix}.
\]
Now subtract 
\[
uA_{8}(M + uA_{9})^{-1} \begin{bmatrix}
0 \\ 0 \\ M + uA_{9}\end{bmatrix}
= 
\begin{bmatrix}
0 \\ 0 \\ uA_{8}\end{bmatrix}
\]
from the second block column and then subtract 
\[
uA_{7}(M + uA_{9})^{-1}\begin{bmatrix}
0 \\ 0 \\ M + uA_{9}
\end{bmatrix} = 
\begin{bmatrix}
0 \\ 0 \\ uA_{7}
\end{bmatrix}
\]
from the first block column to get
\[
\begin{bmatrix}
u(A_{1} - uA_{3}(M + uA_{9})^{-1}A_{7}) & u(A_{2} - uA_{3}(M + uA_{9})^{-1}A_{8}) & 0 \\
u(A_{4} - uA_{6}(M + uA_{9})^{-1}A_{7}) & \id + u(A_{5} - uA_{6}(M + uA_{9})^{-1}A_{8}) & 0 \\
0 & 0 & M + uA_{9}
\end{bmatrix}.
\]
Since $\id + u(A_{5} - uA_{6}(M + uA_{9})^{-1}A_{8})$ is invertible over $R/\mf{m}^{N+1}$, we may apply similar arguments to get rid of the blocks directly above it and directly to the left of it. This gives a matrix of the desired form where $uA_{uA_{1}, u\bs{A}}$ is the upper left block.

It is clear that changing the entries of $uA_1$ changes the entries of $u(A_{1} - uA_{3}(R + uA_{9})^{-1}A_{7})$, and therefore also changes the entries of $uA_{uA_{1}, u\bs{A}}$.  What we have described above is a bijection from $u\Mat_{r}(R/\mf{m}^{N+1})$ to itself, defined by taking $uA_{1}$ to  $uA_{uA_{1}, u\bs{A}}$. This completes the proof of Step (2).
\end{proof}

\subsection{Counting matrices with a given cokernel} Before completing the proof of Step (3) of the outline given in Section \ref{proof_outline}, which finishes the proof of Lemma \ref{l1deg1}, we recall some additional facts.

\begin{lem}\label{SNF} 
Let $R$ be a PID and $X \in \Mat_n(R)$ have rank $r$ over the fraction field of $R$.  There exist $P,Q \in \GL_n(R)$ such that $PXQ = S$ is a diagonal matrix whose diagonal entries $(s_1, s_2, \ldots, s_r,0,\ldots, 0)$ satisfy $s_i \mid s_{i+1}$ for all $1\le i \le r-1$.  Since $\cok(X) \simeq \cok(PXQ) = \cok(S)$, we have 
\[
\cok(X) \cong R/s_1 R \oplus R/s_2 R \oplus \cdots \oplus R/s_r R \oplus R^{n-r}.
\]
Moreover, these $s_i$ are uniquely determined up to multiplication by a unit of $R$, and 
\[
s_1 \cdots s_i = \gcd(i \times i \text{ minors of } X).
\]
We call these $s_1,\ldots, s_r$ the \emph{invariant factors} of $\cok(X)$.
\end{lem}

The following formula for the number of $\bar{X} \in \Mat_{n}(\bF_{q})$ of given rank is well-known.

\begin{lem}\label{rank} 
For any integers $n \geq 1$ and $0 \leq r \leq n$, the number of rank $r$ matrices in $\Mat_{n}(\bF_{q})$ is
\[
q^{n^{2} - (n-r)^{2}}\frac{\prod_{i=1}^{n}(1 - q^{-i}) \prod_{i=n-r+1}^{n}(1 - q^{-i})}{\prod_{i=1}^{n-r}(1 - q^{-i}) \prod_{i=1}^{r}(1 - q^{-i})}.
\]
\end{lem}
We will use the following formula for the number of automorphisms of a finite module over a complete DVR whose residue field is finite. See for example \cite[p. 236]{FW} for a proof.

\begin{lem}\label{Aut} Let $(R, \mf{m})$ be a complete DVR with a finite residue field $R/\mf{m} = \bF_{q}$. Suppose
\[
G \simeq (R/\mf{m}^{e_{1}})^{r_{1}} \times \cdots \times (R/\mf{m}^{e_{k}})^{r_{k}}
\]
for integers $e_{1} > e_{2} > \cdots > e_{k} \geq 1$ and $r_{1}, \dots, r_{k} \geq 1$. Then
\[
|\Aut_{R}(G)| = \prod_{i=1}^{k}q^{-r_{i}^{2}}|\GL_{r_{i}}(\bF_{q})| \prod_{1 \leq i,j \leq k}q^{\min(e_{i},e_{j})r_{i}r_{j}}.
\]
\end{lem}

For clarity, we state Step (3) of the proof outline given in Section \ref{proof_outline} as a separate result.

\begin{lem}  
Let $(R, \mf{m})$ be a complete DVR with a finite residue field $R/\mf{m} = \bF_{q}$ and let $u$ be a uniformizer of $R$. Let $G$ be a finite $R$-module with $r_q(G) := \dim_{\bF_{q}}(G/\mf{m}G) = r$ and $N \in \bZ_{\geq 0}$ satisfy $\mf{m}^{N}G = 0$.  We have
\[
\#\{uA \in u\Mat_{r}(R/\mf{m}^{N+1}) \colon \cok(uA) \simeq G\}  = \frac{q^{Nr^{2} + r^{2}}\prod_{i=1}^{r}(1 - q^{-i})^{2}}{|\Aut_{R}(G)|}.
\]
\end{lem}

As mentioned in Remark \ref{FW_remark}, this result is proven by Friedman and Washington \cite[p.236]{FW}. We give a different proof here that more directly considers the conditions on the entries of a matrix that determine whether or not its cokernel is isomorphic to a particular module $G$.

\begin{proof}
As in the proof of Step (2) of Lemma \ref{l1deg1}, choose representatives $\alpha_1,\ldots, \alpha_q \in R/\mf{m}^{N+1}$ for the equivalence classes of $R/\mf{m}$.  Every element of $R/\mf{m}^{N+1}$ can be expressed uniquely as $a_0 + a_1 u + a_2 u^2 + \cdots + a_N u^N$, where each $a_i$ is equal to some $\alpha_j$.  Let $uA \in \Mat_{r}(R/\mf{m}^{N+1})$ and express each entry of $uA$ in this form.  We have
\[
uA = uA_{1} + u^{2}A_{2} + \cdots + u^{N}A_{N},
\]
where each $A_i$ is an $r \times r$ matrix with entries in $\{\alpha_1,\ldots, \alpha_q\}$.

We want to use Lemma \ref{SNF} to describe the conditions on a matrix in $\Mat_{r}(R/\mf{m}^{N+1})$ that determine whether or not its cokernel is isomorphic to $G$.  However, Lemma \ref{SNF} only applies for matrices with entries in a PID.  Therefore, we take lifts of our matrices to $R$. Choose $u\tilde{A} \in u\Mat_{r}(R)$ to be a fixed lift of $uA \in \Mat_{r}(R/\mf{m}^{N+1})$.  We see that for any $i,j$, the $(i,j)$ entry of $u\tilde{A}$ is congruent to the $(i,j)$ entry of $uA$ modulo $\mf{m}^{N+1}$.  Consider the projection map $\pi \colon R \ra R/\mf{m}^{N+1}$ and suppose $\alpha_1',\ldots, \alpha_q'$ satisfy $\pi(\alpha_i') = \alpha_i$.  Note that $\alpha_1',\ldots, \alpha_q'$ are representatives of the equivalences classes of $R/\mf{m}$.  Expressing each entry of $u\tilde{A}$ in terms of its $u$-adic digit expansion, we can write
\[
u\tilde{A} = uA_{1} + u^{2}A_{2} + \cdots + u^{N}A_{N} + u^{N+1}A_{N+1} + \cdots 
\]
where each $A_i$ is an $r \times r$ matrix with entries in $\{\alpha_1',\ldots, \alpha_q'\}$.  For ease of notation, we identify $A_i$ with its image in $\Mat_{n}(\bF_{q})$ under the map defined by reducing each entry modulo $\mf{m}$.  When we refer to the rank of the matrix $A_i$ we always mean the rank of this matrix in $\Mat_n(\bF_q)$.

We now apply Lemma \ref{SNF} to $u\tilde{A}$ to determine the conditions on the entries of $A_{1}, A_{2}, A_{3}, \dots$ that imply $\cok(u\tilde{A}) \simeq G$. Since  $u^N G = 0$, Lemma \ref{SNF} implies that $\cok(uA) \simeq G$ if and only if $\cok(u\tilde{A}) \simeq G$. Therefore, these conditions are independent of the choice of $A_{N+1}, A_{N+2},\ldots$.  We count choices of $A_1,\ldots, A_N$ for which $\cok(u\tilde{A}) \simeq G$, completing the proof.

Since $R$ is a PID with unique maximal ideal $\mf{m} = (u)$, by the classification of modules over a PID there are unique integers $e_{1} > e_{2} > \cdots > e_{k} \geq 1$ and $r_{1}, \dots, r_{k} \geq 1$ such that 
\[
G \simeq (R/\mf{m}^{e_{1}})^{r_{1}} \times \cdots \times (R/\mf{m}^{e_{k}})^{r_{k}} = (R/u^{e_{1}}R)^{r_{1}} \times \cdots \times (R/u^{e_{k}}R)^{r_{k}}.
\]
Since $r = r_{q}(G)$, we have $r_1 + \cdots + r_k = r$.  By assumption, $u^{N}G = 0$, which is equivalent to $N \geq e_{1}$. 

The invariant factors of $\cok(u\tilde{A})$ are only determined up to multiplication by a unit, so we can assume each one is of the form $u^m$ for some $m\ge 1$.  We order the invariant factors by these exponents since $u^{m_1} \mid u^{m_2}$ if and only if $m_1\le m_2$.  By Lemma \ref{SNF}, the smallest invariant factor of $\cok(u\tilde{A})$ is equal to the greatest common divisor of the $1 \times 1$ minors of $u\tilde{A}$.  Therefore, $u^{e_k}$ is the smallest invariant factor of $\cok(u\tilde{A})$ if and only if $A_{1} = A_{2} = \cdots = A_{e_{k-1}} = 0$ and $A_{e_k}$ is nonzero.  

Our next goal is to count the invariant factors of this smallest size.\\
\noindent {\bf Claim}: Suppose that $A_{1} = A_{2} = \cdots = A_{e_{k-1}} = 0$.  Then $\cok(u\tilde{A})$ has exactly $r_{k}$ invariant factors $u^{e_{k}}$ if and only if $A_{e_{k}}$ has rank $r_{k}$.

\begin{proof}[Proof of Claim]
The rank of $A_{e_k}$ is the largest $m$ such that there exists a nonzero $m \times m$ minor of $A_{e_k}$.  Suppose $i \in [1,r]$.  By Lemma \ref{SNF} applied to $u\tilde{A}$, we have $\ s_1\cdots s_i = \gcd(i\times i \text{ minors of } u\tilde{A})$. As explained above, since $A_{1} = A_{2} = \cdots = A_{e_{k}-1} = 0$, if $A_{e_k} \neq 0$, then $u^{e_k} = s_1$ is the smallest invariant factor of $\cok(u\tilde{A})$.  Since $s_1 \mid \cdots \mid s_i$, we see that $(u^{e_k})^i \mid s_1\cdots s_i$.

If there exists a nonzero $r_{k} \times r_{k}$ minor of $A_{e_{k}}$, then there is an $r_{k} \times r_{k}$ minor of $u\tilde{A}$ equal to $(u^{e_k})^{r_k}$ times a unit in $R$.  This implies $\gcd(r_k \times r_k \text{ minors of } u\tilde{A}) = (u^{e_k})^{r_k}$, and therefore $\cok(u\tilde{A})$ has at least $r_k$ invariant factors equal to $u^{e_k}$.

By the same reasoning, if there is a nonzero $(r_{k}+1) \times (r_{k}+1)$ minor of $A_{e_k}$, then $\cok(u\tilde{A})$ has at least $r_k+1$ invariant factors equal to $u^{e_k}$.  Therefore, $\cok(u\tilde{A})$ has exactly $r_k$ invariant factors equal to $u^{e_k}$ if and only if $A_{e_{k}} \in \Mat_{r}(\bF_{q})$ has rank $r_{k}$. 
\end{proof}

Given $A_{1} = A_{2} = \cdots = A_{e_{k}-1} = 0$ and $A_{e_{k}}$ with rank $r_{k}$, we find the constraints on $A_{e_{k}+1}, A_{e_{k}+2}, \ldots$ that determine whether $\cok(u\tilde{A})$ has no invariant factors between $u^{e_k}$ and $u^{e_{k-1}}$ and exactly $r_{k-1}$ invariant factors equal to $u^{e_{k-1}}$.  Since $A_{e_{k}}$ has rank $r_{k}$, there exist $P_{A_{e_{k}}}, Q_{A_{e_{k}}} \in \GL_{r}(\bF_{q})$ such that
\[
P_{A_{e_{k}}}A_{e_{k}}Q_{A_{e_{k}}} = \begin{bmatrix}
0 & 0 \\
0 & I_{r_k}
\end{bmatrix}.
\]
Let $P_{A_{e_{k}}}'$ be an arbitrary lift of $P_{A_{e_{k}}}$ to $\GL_{r}(R)$ and $Q_{A_{e_{k}}}'$ be an arbitrary lift of $Q_{A_{e_{k}}}$ to $\GL_{r}(R)$.  We have 
\[
\cok(u\tilde{A}) \simeq \cok(P_{A_{e_{k}}}'u\tilde{A}Q_{A_{e_{k}}}') = \cok\left(u^{e_k} A_{e_k}'+  u^{e_k+1} A_{e_k+1}' + \cdots \right),
\]
where $A_j' = P_{A_{e_{k}}}' A_{j} Q_{A_{e_{k}}}'$.

We claim that the next smallest size of an invariant factor of $\cok(u\tilde{A})$ after $u^{e_k}$ is $u^{e_{k-1}}$ if and only if $A_{e_k+1}',\ldots, A_{e_{k-1}-1}'$ have all of their entries in the top left $(r-r_k) \times (r-r_k)$ corner equal to $0$ and $A_{e_{k-1}}'$ has a nonzero entry in its top left $(r-r_k) \times (r-r_k)$ corner.

Suppose $j \in [e_k+1, e_{k-1}-1]$.  If $A_j'$ has a nonzero entry in its top left $(r-r_k) \times (r-r_k)$ corner, then there is an $(r_k+1) \times (r_k+1)$ minor of $u \tilde{A}$ equal to $u^j (u^{e_k})^{r_k}$ times a unit in $R$.  Lemma \ref{SNF} now implies that the next smallest invariant factor of $\cok(u\tilde{A})$ after $u^{e_{k}}$ has size at most $u^j$.

Now suppose $A_{e_k+1}',\ldots, A_{e_{k-1}-1}'$ have all of their entries in the top left $(r-r_k) \times (r-r_k)$ corner equal to $0$. Every $(r_k + 1) \times (r_k + 1)$ minor of $P_{A_{e_{k}}}' u\tilde{A}Q_{A_{e_{k}}}'$ is divisible by $u^{e_{k-1}} (u^{e_k})^{r_k}$.  Lemma \ref{SNF} implies that the next smallest invariant factor of $\cok(u\tilde{A})$ after $u^{e_{k}}$ is at least $u^{e_{k-1}}$.  If there is a nonzero entry in the top left $(r-r_k) \times (r-r_k)$ corner of $A_{e_{k-1}}'$, then there is an $(r_k+1) \times (r_k+1)$ minor of $u \tilde{A}$ equal to $u^{e_{k-1}} (u^{e_k})^{r_k}$ times a unit in $R$.  If every entry in the top left $(r-r_k) \times (r-r_k)$ corner of $A_{e_{k-1}}'$ is $0$, then every $(r_k+1) \times (r_k+1)$ minor of $u \tilde{A}$ is divisible by $u^{e_{k-1}+1} (u^{e_k})^{r_k}$.  Thus, this next smallest invariant factor of $\cok(u\tilde{A})$ is equal to $u^{e_{k-1}}$ if and only if there is a nonzero entry in the top left $(r-r_k) \times (r-r_k)$ corner of $A_{e_{k-1}}'$.

Suppose that the second smallest invariant factor of $\cok(u\tilde{A})$ is equal to $u^{e_{k-1}}$.  We now determine the number of invariant factors of this size.  For any $t$, the greatest common divisor of the $(r_k + t) \times (r_k+t)$ minors of $P_{A_{e_{k}}}' u\tilde{A} Q_{A_{e_{k}}}'$ is divisible by $(u^{e_k})^{r_k} (u^{e_{k-1}})^t$.  By Lemma \ref{SNF}, there are at least $t$ invariant factors of $\cok(u\tilde{A})$ equal to $u^{e_{k-1}}$ if and only if there is an $(r_k+t) \times (r_k+t)$ minor of $P_{A_{e_{k}}}' u\tilde{A} Q_{A_{e_{k}}}'$ equal to $(u^{e_k})^{r_k} (u^{e_{k-1}})^t$ times a unit in $R$.

Consider the top left $(r-r_k) \times (r-r_k)$ corner of $A_{e_{k-1}}'$.  Suppose it has rank $t$.  This matrix has a nonzero $t \times t$ minor, but each of its $(t+1) \times (t+1)$ minors is $0$.  Therefore, there is a $(r_k+t) \times (r_k+t)$ minor of $P_{A_{e_{k}}}' u\tilde{A} Q_{A_{e_{k}}}'$ equal to $(u^{e_k})^{r_k} (u^{e_{k-1}})^t$ times a unit in $R$, but no $(r_k+t+1) \times (r_k+t+1)$ minor of $P_{A_{e_{k}}}' u\tilde{A} Q_{A_{e_{k}}}'$ is equal to $(u^{e_k})^{r_k} (u^{e_{k-1}})^{t+1}$ times a unit in $R$.  We conclude that $\cok(u\tilde{A})$ has exactly $r_{k-1}$ invariant factors equal to $u^{e_{k-1}}$ if and only if $t = r_{k-1}$.  As above, we multiply by appropriate $P_{A_{e_{k-1}}}', Q_{A_{e_{k-1}}}' \in \GL_r(R)$.  We then repeat this argument for the remaining invariant factors of $G$.

When considering the matrix that determines the number of invariant factors of $\cok(u\tilde{A})$ equal to $u^{e_1}$, we need the top left $r_1 \times r_1$ piece of an $(r_1+r_2) \times (r_1+r_2)$ submatrix to have rank $r_1$.  There are no conditions on the remaining $2 r_1 r_2 + r_2^2$ entries of this submatrix.  Similarly, when we consider the matrix that determines the number of invariant factors of $\cok(u\tilde{A})$ equal to $u^{e_2}$, we need the top left $(r_1+r_2) \times (r_1+r_2)$ piece of an $(r_1+r_2+r_3) \times (r_1+r_2+r_3)$ submatrix to have rank $r_2$.  There are no conditions on the remaining $2 (r_1 + r_2)r_3 + r_3^2$ entries of this submatrix.  Continuing in this way, we see that the total number of choices of $uA \in \Mat_r(R/\mf{m}^{N+1})$ for which $\cok(u\tilde{A}) \simeq G$ is equal to
\[
\bigg(\prod_{i=0}^{k-1} \#\left\{\begin{array}{c}
X \in \Mat_{r-(r_{k} + r_{k-1} + \cdots + r_{k-i+1})}(\bF_{q}) \colon \\
\rk(X) = r_{k-i}
\end{array}\right\}\bigg)
\left(
q^{\sum_{i=1}^k (N-e_i)(r_i^2 + 2r_i \sum_{1\le j<i} r_j)}\right).
\]
We have
\[
q^{\sum_{i=1}^k (N-e_i)(r_i^2 + 2r_i \sum_{1\le j<i} r_j)}
= q^{N r^2 - \sum_{i=1}^k e_i r_i^2 - 2\sum_{1\le j< i \le k} e_i r_i r_j}
 =  \frac{q^{Nr^{2}}}{\prod_{1 \leq i, j \leq k} q^{\min(e_{i}, e_{j})r_{i}r_{j}} }
\]
since $r = r_{1} + \cdots + r_{k}$. 

Applying Lemma \ref{rank} shows that
\begin{align*}
\prod_{i=0}^{k-1} & \#\left\{\begin{array}{c}
X \in \Mat_{r-(r_{k} + r_{k-1} + \cdots + r_{k-i+1})}(\bF_{q}) \colon \\
\rk(X) = r_{k-i}
\end{array}\right\} \\
&= q^{r^{2}}\frac{(1 - q^{-1})^{2}(1 - q^{-2})^{2} \cdots (1 - q^{-r})^{2}}{(1 - q^{-1})(1 - q^{-2}) \cdots (1 - q^{-r_{1}}) \cdot \cdots \cdot (1 - q^{-1})(1 - q^{-2}) \cdots (1 - q^{-r_{k}})} \\ 
&= q^{r^{2}}\frac{\prod_{i=1}^{r}(1 - q^{-i})^{2}}{q^{-(r_{1}^{2} + \cdots + r_{k}^{2})}|\GL_{r_{1}}(\bF_{q})| \cdots |\GL_{r_{k}}(\bF_{q})|}.
\end{align*}

Therefore, we have
\begin{align*}
& \bigg(\prod_{i=0}^{k-1} \#\left\{\begin{array}{c}
X \in \Mat_{r-(r_{k} + r_{k-1} + \cdots + r_{k-i+1})}(\bF_{q}) \colon \\
\rk(X) = r_{k-i}
\end{array}\right\}\bigg)
\left(
q^{\sum_{i=1}^k (N-e_i)(r_i^2 + 2r_i \sum_{1\le j<i} r_j)}\right)\\
&= \frac{q^{Nr^{2} + r^{2}}\prod_{i=1}^{r}(1 -q^{-i})^{2}}{\prod_{i=1}^{k}q^{-r_{i}^{2}}|\GL_{r_{i}}(\bF_{q})| \prod_{1 \leq i, j \leq k}q^{\min(e_{i}, e_{j})r_{i}r_{j}}} \\
&= \frac{q^{Nr^{2} + r^{2}}\prod_{i=1}^{r}(1 - q^{-i})^{2}}{|\Aut_{R}(G)|},
\end{align*}
where in the last step we applied Lemma \ref{Aut}.
\end{proof}

\section{The Proof of Theorem \ref{main3}}\label{sec_main_general} 

In this section, we use the special case of Theorem \ref{main3} where $l=1$ to prove Theorem \ref{main3} in general.  We first recall some notation and assumptions.  We are given
\begin{itemize}
	\item monic polynomials $P_{1}(t), \dots, P_{l}(t) \in \bZ_{p}[t]$ of degree at most $2$ whose images in $\bF_{p}[t]$ are distinct and irreducible and a finite module $G_{j}$ over $\bZ_{p}[t]/(P_{j}(t))$ for each $1 \leq j \leq l$, 
	\item $N \in \bZ_{\geq 0}$ such that $p^{N}G_j = 0$ for each $1\le j \le l$, and
	\item $\bar{X} \in \Mat_{n}(\bF_{p})$ such that for each $1\le j \le l$,
	\[
	\dim_{\bF_{q_{j}}}\left(\cok(P_{j}(\bar{X}))\right) = r_{q_{j}}(G_{j}),
	\]
	where $q_j = p^{\deg(P_j)}$.
	\end{itemize}
We count $X \in \Mat_{n}(\bZ/p^{N+1}\bZ)$ such that for each $1 \leq j \leq l$ we have $\cok(P_{j}(X)) \simeq G_{j}$.

\begin{proof}[Proof of Theorem \ref{main3}] 
We use an argument similar to two steps of the proof outline given for Lemma \ref{l1deg1} in Section \ref{proof_outline}.  There exists $\bar{Q} \in \GL_{n}(\bF_{p})$ such that
\[
\bar{Q}^{-1}\bar{X}\bar{Q} = \begin{bmatrix}
J_{1} & 0 & \cdots & 0 & 0 \\
0 & J_{2} & \cdots & 0 & 0 \\
\vdots & \vdots & \cdots & \vdots & \vdots \\
0 & 0 & \cdots & J_{l} & 0  \\
0 & 0 & \cdots & 0 & M
\end{bmatrix},
\]
where each eigenvalue of $J_{j}$ over $\overline{\bF}_p$ is a root of $P_j(t)$ and $M$ is a square matrix over $\bF_{p}$ with no eigenvalue over $\overline{\bF}_{p}$ that is a root of any of $P_{1}(t), \dots, P_{l}(t)$.  Note that if $G_j$ is trivial, then $J_j$ is the empty matrix.

Fix a lift $Q \in \GL_{n}(\bZ/p^{N+1}\bZ)$ of $\bar{Q} \in \GL_{n}(\bF_{p})$. Given a lift $X \in \Mat_{n}(\bZ/p^{N+1}\bZ)$ of $\bar{X} \in \Mat_{n}(\bF_{p})$, the matrix $Q^{-1}XQ$ is a lift of $\bar{Q}^{-1}\bar{X}\bar{Q}$. On the other hand, if $Y \in \Mat_{n}(\bZ/p^{N+1}\bZ)$ is a lift of $\bar{Q}^{-1}\bar{X}\bar{Q}$, then $QYQ^{-1}$ is a lift of $\bar{X}$.  Note that $P_{j}(Q^{-1}XQ) = Q^{-1}P_{j}(X)Q$.  As in Step (1) of the proof of Lemma \ref{l1deg1} described in Section \ref{proof_outline}, taking advantage of this bijection between the lifts of $\bar{X}$ and the lifts of $\bar{Q}^{-1}\bar{X}\bar{Q}$, it is enough to prove Theorem \ref{main2} for the case where 
\[
\bar{X} = \begin{bmatrix}
J_{1} & 0 & \cdots & 0 & 0 \\
0 & J_{2} & \cdots & 0 & 0 \\
\vdots & \vdots & \cdots & \vdots & \vdots \\
0 & 0 & \cdots & J_{l} & 0  \\
0 & 0 & \cdots & 0 & M
\end{bmatrix}.
\]

Fix $j \in [1,l]$ and let $R_{j} := (\bZ/p^{N+1}\bZ)[t]/(P_{j}(t))$.  Over $R_{j}/pR_{j} = \bF_{p}[t]/(P_{j}(t))$, the matrix $J_{j} - \bar{t}\id$ is not invertible, while for any $k \neq j$, the matrix $J_{k} - \bar{t}\id$ is invertible.  The matrix $M - \bar{t}\id$ is also invertible.

Any lift $X \in \Mat_{n}(\bZ/p^{N+1}\bZ)$ of $\bar{X}$ is of the form
\[
X = 
\begin{bmatrix}
pA_{11} + J_{1} & pA_{12} & \cdots & pA_{1l} & pA_{1,l+1} \\
pA_{21} & pA_{22} + J_{2} & \cdots & pA_{2l} & pA_{2,l+1} \\
\vdots & \vdots & \cdots & \vdots & \vdots \\
pA_{l1} & pA_{l2} & \cdots & pA_{ll} + J_{l} & pA_{l,l+1}  \\
pA_{l+1,1} & pA_{l+1,2} & \cdots & pA_{l+1,l} & pA_{l+1,l+1} + M
\end{bmatrix},
\]
where the $pA_{st}$ are matrices over $\bZ/p^{N+1}\bZ$ all of whose entries are $0$ modulo $p$, and each of $J_1,\ldots, J_l, M$ is the unique lift to $\bZ/p^{N+1}\bZ$ with entries in $\{0,1,\ldots, p-1\}$ of the corresponding matrix over $\bF_p$. Let $n_{1}, \dots, n_l, n_{l+1} \in \bZ_{\geq 0}$ so that $pA_{ii} \in p\Mat_{n_{i}}(\bZ/p^{N+1}\bZ)$ for $1 \leq i \leq l+1$. In particular, $n = n_{1} + \cdots + n_{l} + n_{l+1}$. 

A key idea in this argument is to work with expansions of elements of $\bZ/p^{N+1}\bZ$ in terms of powers of $p$, and to use these expansions of elements to give similar expansions for matrices.  Each entry of $pA_{st}$ is of the form $a_{1}p + a_{2}p^{2} + \cdots + a_{N-1}p^{N-1} + a_{N}p^{N}$, where $a_{i} \in \{0,1,\ldots, p-1\}$.  Without any constraints there are $p^{Nn^{2}}$ total lifts $X$ of $\bar{X}$. Recall that by Lemma \ref{Lee}, we have $\cok(P_{j}(X)) \simeq \cok_{R_{j}}(X - \bar{t}I_{n})$.

We choose a sequence of (block) row and column operations from Lemma \ref{elem} to transform the matrix 
\begin{equation}\label{XtIn}
X - \bar{t}I_{n} = 
\begin{bmatrix}
pA_{11} + J_{1} - \bar{t}\id & pA_{12} & \cdots & pA_{1l} & pA_{1,l+1} \\
pA_{21} & pA_{22} + J_{2} - \bar{t}\id & \cdots & pA_{2l} & pA_{2,l+1} \\
\vdots & \vdots & \ddots & \vdots & \vdots \\
pA_{l1} & pA_{l2} & \cdots & pA_{ll} + J_{l} - \bar{t}\id & pA_{l,l+1}  \\
pA_{l+1,1} & pA_{l+1,2} & \cdots & pA_{l+1,l} & pA_{l+1,l+1} + M - \bar{t}\id
\end{bmatrix}
\end{equation}
into the block diagonal matrix whose blocks are given by 
\[
pA_{11} + J_{1} - \bar{t}\id + p^{2}B_{1}^{(j)} + \bar{t} p^2C_{1}^{(j)}, \dots,  pA_{ll} + J_{l} - \bar{t}\id + p^{2}B_{l}^{(j)} + \bar{t} p^2 C_{l}^{(j)},  pA_{l+1,l+1} + M - \bar{t}\id,
\]
where $p^2 B_{i}^{(j)}, p^2 C_{i}^{(j)} \in p^2 \Mat_{n_i}(\bZ/p^{N+1}\bZ)$ depend on $pA_{st}$ with $1 \leq s, t \leq l+1$ except they do not depend on $pA_{ii}$.  We conclude that $X-\bar{t} I_n$ and this block diagonal matrix have isomorphic cokernels over $R_j$.

Suppose $i,j\in [1,l]$ are distinct.  Since $J_{i} - \bar{t}\id$ is invertible over $R_j/p R_j$, we see that $pA_{ii} + J_{i} - \bar{t}\id + p^{2}B_{i}^{(j)} + \bar{t}p^2 C_{i}^{(j)}$ is invertible over $R_j$.  Similarly, $pA_{l+1,l+1} + M - \bar{t}\id$ is invertible over $R_j$.  This implies
\begin{align*}
\cok(P_{j}(X)) &\simeq \cok_{R_{j}}(X - \bar{t}I_{n}) \\
&\simeq \cok_{R_{j}}(pA_{jj} + J_{j} - \bar{t}\id + p^{2}B_{j}^{(j)} + \bar{t} p^2 C_{j}^{(j)}) \\
&= \cok_{R_{j}}(J_{j} + pA_{jj} + p^{2}B_{j}^{(j)} - \bar{t}(\id  - p^{2}C_{j}^{(j)})).
\end{align*}
Multiplying by an invertible matrix does not change the isomorphism class of the cokernel, so
\begin{eqnarray*}
\cok_{R_{j}}(J_{j} + pA_{jj} + p^{2}B_{j}^{(j)} - \bar{t}(\id  - p^{2}C_{j}^{(j)}))
& \simeq & \cok_{R_{j}}((\id  - p^{2}C_{j}^{(j)})^{-1}(J_{j} + pA_{jj} + p^{2}B_{j}^{(j)}) - \bar{t}\id)\\
&\simeq & \cok\big(P_{j}\big((\id  - p^{2}C_{j}^{(j)})^{-1}(J_{j} + pA_{jj} + p^{2}B_{j}^{(j)})\big)\big).
\end{eqnarray*}

Fix any choices for $\{pA_{st}\colon 1 \le s,t \le l+1,\ s\neq t\}$ and also fix a choice of $pA_{l+1,l+1}$.  Note that there are $p^{N(n^2 - (n_1^2 + \cdots + n_l^2))}$ total possible choices for these matrices.  The matrices $p^2 B_i^{(j)},p^2 C_i^{(j)} \in p^2 \Mat_{n_i}(\bZ/p^{N+1}\bZ))$ play an important role in our proof.  Given choices for the $pA_{st}$ above and for $pA_{l+1,l+1}$ we can think of $p^2 B_i^{(j)}$ and $p^2 C_i^{(j)}$ as being functions of $(pA_{11},\ldots, pA_{i-1,i-1},pA_{i+1,i+1},\ldots, pA_{ll})$.  We usually do not usually include this dependence in the notation for $p^2 B_i^{(j)}$ and $p^2 C_i^{(j)}$ because it would things much harder to read, but in a few cases we include this additional notation for emphasis.

Consider the function $\Phi_N$ that takes $(pA_{11}, pA_{22},\ldots, pA_{ll})$ to 
\[
\left(
(\id  - p^{2}C_{1}^{(1)})^{-1}(J_{1} + pA_{11} + p^{2}B_{1}^{(1)}),
\cdots,
(\id  - p^{2}C_{l}^{(l)})^{-1}(J_{l} + pA_{ll} + p^{2}B_{l}^{(l)})
\right).
\]
So $\Phi_N$ is a map from $p\Mat_{n_{1}}(\bZ/p^{N+1}\bZ) \times \cdots \times p\Mat_{n_{l}}(\bZ/p^{N+1}\bZ)$ to 
\[
\left\{\begin{array}{c}
Y_1 \in \Mat_{n_{1}}(\bZ/p^{N+1}\bZ) \colon\\
Y_1 \equiv J_{1} \pmod{p}
\end{array}\right\} \times \cdots \times
\left\{\begin{array}{c}
Y_l \in \Mat_{n_{l}}(\bZ/p^{N+1}\bZ) \colon\\
Y_l \equiv J_{l} \pmod{p}
\end{array}\right\}.
\]

\noindent{\bf Claim}: $\Phi_N$ is a bijection.

Assuming the claim for now, we complete the proof of Theorem \ref{main3}.  By the discussion above, $\cok(P_j(X)) \simeq G_j$ for each $j$ if and only if $\cok\big(P_{j}\big((\id  - p^{2}C_{j}^{(j)})^{-1}(J_{j} + pA_{jj} + p^{2}B_{j}^{(j)})\big)\big) \simeq G_j$ for each $j$.  Since $\Phi_N$ is a bijection, the number of choices for $(pA_{11},\ldots, pA_{ll})$ such that $\cok(P_j(X)) \simeq G_j$ for each $j$ is equal to 
\[
\# \left\{\begin{array}{c}
(Y_{1}, \ldots, Y_{l}) \in \Mat_{n_{1}}(\bZ/p^{N+1} \bZ) \times \cdots \times \Mat_{n_{l}}(\bZ/p^{N+1} \bZ) : \\
\cok(P_{j}(Y_{j})) \simeq G_{j} \text{ and } Y_{j} \equiv J_{j} \pmod{p} \text{ for } \text{for } 1 \leq j \leq l
\end{array}\right\}.
\]
It is clear that this is equal to
\[
\prod_{j=1}^l \# \left\{\begin{array}{c}
Y_j \in \Mat_{n_{j}}(\bZ/p^{N+1} \bZ) : \\
\cok(P_{j}(Y_{j})) \simeq G_{j} \text{ and } Y_{j} \equiv J_{j} \pmod{p} 
\end{array}\right\}.
\]
By the $l=1$ case of Theorem \ref{main2}, this is equal to 
\[
\prod_{j=1}^{l} p^{Nn_{j}^{2}}\frac{q_{j}^{r_{q_{j}}(G_{j})^{2}}\prod_{i=1}^{r_{q_{j}}(G_{j})}(1 - q_{j}^{-i})^{2}}{|\Aut_{\bZ_p[t]/(P_j(t))}(G_{j})|}.
\]
Multiplying by $p^{N(n^2 - (n_1^2 + \cdots + n_l^2))}$ to account for all possible choices of $\{pA_{st}\colon 1 \le s,t \le l+1,\ s\neq t\}$ and $pA_{l+1,l+1}$ completes the proof of Theorem \ref{main3}.

We now need only prove that $\Phi_N$ is a bijection.  Since $\Phi_N$ is a map between finite sets of the same size, we need only prove that it is surjective.  We define a sequence of maps $\Phi_1,\ldots, \Phi_{N-1}, \Phi_N$ and prove that each one is a bijection.  Let $k \in [1,N]$.  Our next goal is to define the map $\Phi_k$.
\begin{itemize}
\item For each element of $\{pA_{st}\colon 1 \le s,t \le l+1,\ s\neq t\}$, let $pA_{st}^{[k]}$ be the matrix over $\bZ/p^{k+1}\bZ$ such that $pA_{st} \equiv pA_{st}^{[k]} \pmod{p^{k+1}}$.   
\item Let $pA_{l+1,l+1}^{[k]}$ be the matrix over $\bZ/p^{k+1}\bZ$ such that $pA_{l+1,l+1} \equiv pA_{l+1,l+1}^{[k]} \pmod{p^{k+1}}$. 
\end{itemize}

For each $j \in [1,l]$, choose a matrix $pA_{jj} \in \Mat_{n_j}(\bZ/p^{k+1}\bZ)$.  We construct a matrix analogous to the one given in \eqref{XtIn} using $pA_{11},\ldots, pA_{ll}, pA_{l+1,l+1}^{[k]}$ and the elements of $\{pA^{[k]}_{st}\colon 1 \le s,t \le l+1,\ s\neq t\}$.  We then apply the same sequence of (block) row and column operations that we applied to the matrix in \eqref{XtIn}.  This gives a block diagonal matrix whose blocks are given by
\[
pA_{11} + J_{1} - \bar{t}\id + p^{2}B_{1}^{(j,k)} + \bar{t} p^2C_{1}^{(j,k)}, \dots,  pA_{ll} + J_{l} - \bar{t}\id + p^{2}B_{l}^{(j,k)} + \bar{t} p^2 C_{l}^{(j,k)},  pA_{l+1,l+1} + M - \bar{t}\id,
\]
where $p^2 B_i^{(j,k)}, p^2 C_i^{(j,k)} \in p^2 \Mat_{n_i}(\bZ/p^{k+1}\bZ)$ are functions of $(pA_{11},\ldots, pA_{i-1,i-1},pA_{i+1,i+1},\ldots, pA_{ll})$.  A key thing to note is that for any $1\le k'\le k\le N$, if $pA_{jj} \equiv pA_{jj}' \pmod{p^{k'+1}}$ for each $j \in [1,l]$, then
\begin{align*}
& p^2 B_i^{(j,k)}(pA_{11},\ldots, pA_{i-1,i-1},pA_{i+1,i+1},\ldots, pA_{ll}) \\
\equiv \ \ \ &  p^2 B_i^{(j,k')}(pA'_{11},\ldots, pA'_{i-1,i-1},pA'_{i+1,i+1},\ldots, pA'_{ll}) \pmod{p^{k'+1}}.
\end{align*}
An analogous result holds for the matrices $p^2 C_i^{(j,k)}$ and $p^2 C_i^{(j,k')}$. 

We define a map
\[
\Phi_k \colon p\Mat_{n_{1}}(\bZ/p^{k+1}\bZ) \times \cdots \times p\Mat_{n_{l}}(\bZ/p^{k+1}\bZ) 
\ra
\left\{\begin{array}{c}
Y_1 \in \Mat_{n_{1}}(\bZ/p^{k+1}\bZ) \colon\\
Y_1 \equiv J_{1} \pmod{p}
\end{array}\right\} \times \cdots \times
\left\{\begin{array}{c}
Y_l \in \Mat_{n_{l}}(\bZ/p^{k+1}\bZ) \colon\\
Y_l \equiv J_{l} \pmod{p}
\end{array}\right\}
\]
that takes $(pA_{11}, pA_{22},\ldots, pA_{ll})$ to 
\[ \left(
(\id - p^{2}C_{1}^{(1,k)})^{-1}(J_{1} + pA_{11}+ p^{2}B_{1}^{(1,k)}), 
\cdots,
(\id - p^{2}C_{l}^{(l,k)})^{-1}(J_{l}+ pA_{ll}+ p^{2}B_{l}^{(l,k)})
\right).
\]

Now that we have defined $\Phi_1,\ldots, \Phi_N$, we see that they are compatible with reducing the inputs modulo powers of $p$.  More precisely, for any $1\le k' \le k \le N$, if $pA_{jj} \equiv pA_{jj}' \pmod{p^{k'+1}}$ for every $j \in [1,l]$, then 
\[
\Phi_k(pA_{11}, pA_{22},\ldots, pA_{ll}) \equiv \Phi_{k'}(pA'_{11}, pA'_{22},\ldots, pA_{ll}') \pmod{p^{k'+1}}.
\]

We now prove that $\Phi_1$ is a bijection.  We have
\[
(\id  - p^{2}C_{j}^{(j,1)})^{-1}(J_{j} + pA_{jj}+ p^{2}B_{j}^{(j,1)}) \equiv J_j  + p A_{jj} \pmod{p^2},
\]
so it is clear that $\Phi_1$ is a bijection.

We now assume that $\Phi_{k-1}$ is a bijection and prove that $\Phi_k$ is a surjection.  Choose $(Y_1,\ldots, Y_l)$ such that for each $j$, we have $Y_j \in \Mat_{n_j}(\bZ/p^{k+1}\bZ)$ with $Y_j \equiv J_j\pmod{p}$.  Such a matrix can be written uniquely as
\[
Y_j = J_j + p S_j^{[1]} + p^2 S_j^{[2]} + \cdots + p^k S_j^{[k]},
\]
where each $S_j^{[i]}$ is an $n_j \times n_j$ matrix with entries in $\{0,1,\ldots, p-1\}$.  Define $Y_j' \in \Mat_{n_j}(\bZ/p^k \bZ)$ by
\[
Y_j' = J_j+ p S_j^{[1]} + p^2 S_j^{[2]} + \cdots + p^{k-1} S_j^{[k-1]},
\]
so $Y_j \equiv Y_j' \pmod{p^k}$.

Since $\Phi_{k-1}$ is a bijection, there exists $(pA_{11}',\ldots, pA_{ll}') \in p\Mat_{n_{1}}(\bZ/p^{k}\bZ) \times \cdots \times p\Mat_{n_{l}}(\bZ/p^{k}\bZ)$ such that for each $j \in [1,l]$,
\[
(\id  - p^{2}C_{j}^{(j,k-1)})^{-1}(J_{j} + pA_{jj}' + p^{2}B_{j}^{(j,k-1)}) = J_j + p S_j^{[1]} + p^2 S_j^{[2]} + \cdots + p^{k-1} S_j^{[k-1]} =Y_j'.
\]
There are unique $n_j \times n_j$ matrices $T_j^{[k']}$ with entries in $\{0,1,\ldots, p-1\}$ such that 
\[
pA_{jj}' = p T_j^{[1]} + p^2 T_j^{[2]} + \cdots + p^{k-1} T_j^{[k-1]}.
\]
Let $pA_j^*(T_j^{[k]}) \in p\Mat_{n_{j}}(\bZ/p^{k+1}\bZ)$ be defined by
\[
pA_j^*(T_j^{[k]}) =  \left(p T_j^{[1]} + p^2 T_j^{[2]} + \cdots + p^{k-1} T_j^{[k-1]}\right) + p^k T_j^{[k]},
\]
where $T_j^{[k]}$ is an $n_j \times n_j$ matrix with entries in $\{0,1,\ldots, p-1\}$.

We claim that there exist $T_1^{[k]},\ldots, T_l^{[k]}$ such that 
\[
\Phi_k\left(pA_1^*(T_1^{[k]}),\ldots, pA_l^*(T_l^{[k]})\right) =  (Y_1,\ldots, Y_l).
\]
For any choice of $(T_1^{[k]},\ldots, T_l^{[k]})$, since for each $j\in [1,l]$ we have $pA_j^*(T_j^{[k]}) \equiv pA_{jj}' \pmod{p^k}$, we see that
\[
\Phi_k(pA_1^*(T_1^{[k]}),\ldots, pA_l^*(T_l^{[k]})) \equiv
\Phi_{k-1}(pA_{11}',\ldots, pA_{ll}') = (Y_1',\ldots, Y_l') \pmod{p^k}.
\]
So there exist $S_1^*,\ldots, S_l^*$, where each $S_j^*$ is an $n_j \times n_j$ matrix with entries in $\{0,1,\ldots, p-1\}$, such that 
\[
(\id  - p^{2}C_{j}^{(j,k)})^{-1}(J_{j} + p A_j^*(T_j^{[k]})+ p^{2}B_{j}^{(j,k)}) = 
J_j + p S_j^{[1]} + p^2 S_j^{[2]} + \cdots + p^{k-1} S_j^{[k-1]} + p^k S_j^* \in \Mat_{n_j}(\bZ/p^{k+1} \bZ).
\]

The crucial observation is that because of the factor of $p^2$,  
\[
p^{2}B_{j}^{(j,k)}(pA_1^*(T_1^{[k]}),\ldots, pA_l^*(T_l^{[k]}))\ \ \  \text{ and}\ \ \ \ p^{2}C_{j}^{(j,k)}(pA_1^*(T_1^{[k]}),\ldots, pA_l^*(T_l^{[k]}))
\] 
depend on $pA_1^*(T_1^{[k]}),\ldots, pA_l^*(T_l^{[k]})$ but they do not depend on $T_1^{[k]}, \ldots, T_l^{[k]}$.  That is, once we have fixed choices for $(T_i^{[1]},\ldots, T_i^{[k-1]})$ for each $i$, the matrices $p^{2}B_{j}^{(j,k)}, p^{2}C_{j}^{(j,k)}$ are determined.  Therefore, we see that each $S_j^*$ depends only on a choice of $T_j^{[k]}$ and not on the choices for $T_i^{[k]}$ where $i \neq j$.  

By definition, a different choice of $T_j^{[k]}$ gives a different matrix $pA_j^*(T_j^{[k]})$.  It is now clear that a different matrix $pA_j^*(T_j^k)$ gives a different matrix $S_j^*$.  That is, the map taking $T_j^{[k]}$ to 
\[
(\id  - p^{2}C_{j}^{(j,k)})^{-1}(J_{j} + p A_j^*(T_j^{[k]})+ p^{2}B_{j}^{(j,k)}) \in
\left\{
\begin{array}{c}
Y_j^* \in \Mat_{n_j}(\bZ/p^{k+1}\bZ) \colon\\
Y_j^* \equiv Y_j' \pmod{p^k}
\end{array}
\right\}
\]
is injective.  Since this is an injective map between finite sets of the same size, it is a bijection.  We conclude that there is a choice of $(T_1^{[k]},\ldots, T_l^{[k]})$ such that for each $j\in [1,l]$, we have 
\[
(\id  - p^{2}C_{j}^{(j,k)})^{-1}(J_{j} + p A_j^*(T_j^{[k]})+ p^{2}B_{j}^{(j,k)}) = Y_j.
\]
Therefore, $\Phi_k$ is a surjection and so, a bijection.  Continuing in this way, we conclude that $\Phi_N$ is a bijection.

\end{proof}

\section*{Acknowledgements}

The second author was supported by NSF Grant DMS 1802281. We thank Yifeng Huang for helpful conversations.  We thank Jungin Lee for several helpful conversations and for contributing a key idea about polynomials of degree $2$.  We thank Kelly Isham for comments related to an earlier draft of this paper.


\begin{thebibliography}{99}
\bibitem[CH2021]{CH}
G. Cheong and Y.Huang,
\emph{Cohen--Lenstra distributions via random matrices over complete discrete valuation rings with finite residue fields}, Illinois J. Math. \textbf{65} (2021), no. 2, 385--415.

\bibitem[CL1983]{CL}
H. Cohen and H. W. Lenstra, Jr.,
\emph{Heuristics on class groups of number fields}, Proceedings of the Journees Arithmetiques held at Noordwijkerhout, the Netherlands, July 11-15, 1983, Lecture
Notes in Mathematics \textbf{1068} (1983), Springer-Verlag, New York, 33--62.

\bibitem[EVW2016]{EVW}
J. Ellenberg, A. Venkatesh, and C. Westerland,
\emph{Homological stability for {H}urwitz spaces and the {C}ohen--{L}enstra conjecture over function fields}, Ann. of Math. (2) \textbf{183} (2016), no. 3, 729--786.

\bibitem[Ev2004]{Evans}
S. Evans,
\emph{Elementary divisors and determinants of random matrices over a local field},
Stochastic Process. Appl. \textbf{102} (2002), no. 1, 89--102.

\bibitem[FW1987]{FW}
E. Friedman and L. C. Washington,
\emph{On the distribution of divisor class groups of curves over a finite field}, Th\'eorie des Nombres (Quebec, PQ, 1987), de Gruyter, Berlin (1989), 227--239.

\bibitem[Lee2022]{LeePaper}
J. Lee, \emph{Joint distribution of the cokernels of random $p$-adic matrices}, in preparation (2022).

\bibitem[VP2021]{VanPeski}
R. Van Peski,
\emph{Limits and fluctuations of $p$-adic random matrix products},
Sel. Math. New Ser. \textbf{27}, 98 (2021).\\
\url{https://doi.org/10.1007/s00029-021-00709-3}

\bibitem[Woo2016]{Woo16}
M. M. Wood,
\emph{Asymptotics for number fields and class groups}, Directions in number theory, 291-339, Assoc. Women Math. Ser., 3, Springer, [Cham], 2016.

\bibitem[Woo2017]{Woo17}
M. M. Wood,
\emph{The distribution of sandpile groups of random graphs}, J. Amer. Math. Soc. \textbf{30} (2017), no. 4, 915--958.

\bibitem[Woo2019]{Woo19}
M. M. Wood,
\emph{Random integral matrices and the {C}ohen--{L}enstra heuristics}, Amer. J. Math. \textbf{141} (2019), no. 2, 383--398.



\end{thebibliography}
\end{document}